\documentclass[a4paper,final,conference]{IEEEtran}
\pdfoutput=1
\usepackage{mathematik}
\renewcommand{\thesection}{\arabic{section}}
\makeatletter
\renewcommand{\section}{\@startsection
  {section}{1}{0mm}{-1.5\baselineskip}{0.5\baselineskip}%
  {\normalfont\normalsize\large\bfseries}%
}
\renewcommand{\subsection}{\@startsection
  {subsection}{2}{0mm}{-1\baselineskip}{0.5\baselineskip}%
  {\normalfont\normalsize\itshape\bfseries}%
}
\makeatother

\title{Decomposition Theorems and Model-Checking for the Modal $\mu$-Calculus}
\author{\IEEEauthorblockN{Mikolaj
    Bojanczyk}%
  \IEEEauthorblockA{University of Warsaw\\%
    \href{mailto:bojan@mimuw.edu.pl}{bojan@mimuw.edu.pl}}%
  \and
  \IEEEauthorblockN{Christoph Dittmann and Stephan Kreutzer}%
  \IEEEauthorblockA{Technical University Berlin\\%
  \{\href{mailto:christoph.dittmann@tu-berlin.de}{christoph.dittmann},
  \href{mailto:stephan.kreutzer@tu-berlin.de}{stephan.kreutzer}\}@tu-berlin.de}}
\date{\today}

\newcommand{\DD}{\mathcal{D}} % Kelly decomposition

\newcommand{\st}{\mathrel : }

\renewcommand{\phi}{\varphi}

\newcommand{\Var}{\mathrm{Var}}
\newcommand{\free}{\mathrm{free}}
\newcommand{\sub}{\mathrm{sub}}
\newcommand{\ind}[1]{#1}
\newcommand{\colored}[3]{\partial_{#2}(#1,#3)}
\newcommand{\closure}{\mathrm{closure}}
\newcommand{\CL}{\mathrm{CL}}
\newcommand{\PT}{\mathrm{PT}}

\newcommand{\ctypeWithIndex}[5]{\mathrm{tp}_{#1,#2}(#3,#4,#5)}

\newcommand{\structureTypeWithIndex}[4]{\mathcal{T}_{#1,#2}(#3,#4)}

\newcommand{\allTypesWithIndex}[2]{\mathcal{T}_{#2}(#1)}
\newcommand{\prio}[3]{\mathrm{prio}_{#1}({#2}\leadsto{#3})}
\newcommand{\strategytargets}{\mathrm{\text{strategy-targets}}}
\newcommand{\strategypreprofile}{\mathrm{preprofile}}
\newcommand{\strategyprofile}{\mathrm{profile}}
\newcommand{\strategyprofiles}{\mathrm{profiles}}
\newcommand{\strategytype}[2]{\mathrm{ptype}_{#1}(#2)}
\newcommand{\defineprofile}[2]{{#1}^{#2}}
\newcommand{\MCgame}[2]{\ensuremath{\mathrm{P}[#1,#2]}}
\newcommand{\MCgameP}[3]{\ensuremath{\mathrm{P}[#1,#2,#3]}}
\newcommand{\guard}{\textup{guard}}
\newcommand{\Bdown}{\mathcal{B}^\downarrow}

\newcommand{\seq}[1]{\overline{#1}}

\begin{document}
\renewenvironment{proof}[1][]{\noindent{%
\ifthenelse{\equal{#1}{}}{{\sl Proof.\ }}{{\sl Proof #1.\ }}%
}}{\hspace*{1em}\nobreak\hfill$\Box$\endtrivlist\addvspace{2ex plus
0.5ex minus0.1ex}}

\maketitle

\begin{abstract}
  We prove a general decomposition theorem for the modal
  $\mu$-calculus $L_\mu$ in the spirit of Feferman and Vaught's
  theorem for disjoint unions. In particular, we show that if a
  structure (i.e., transition system) is composed of two substructures
  $M_1$ and $M_2$ plus edges from $M_1$ to $M_2$, then the formulas
  true at a node in $M$ only depend on the formulas true in the
  respective substructures in a sense made precise below.

  As a consequence we show that the model-checking problem for $L_\mu$
  is fixed-parameter tractable (fpt) on classes of structures of
  bounded Kelly-width or bounded DAG-width. As far as we are aware,
  these are the first fpt results for $L_\mu$ which do not follow from
  embedding into monadic second-order logic.
\end{abstract}

\section{Introduction}

The modal $μ$-calculus $L_μ$, introduced by Dexter Kozen in 1983, is a
well-known logic in the theory of verification that encompasses many
other modal logics.  Among others, propositional dynamic logic (PDL),
linear time logic (LTL) and the full branching time logic (CTL*) have
embeddings into $L_μ$. See e.g.~\cite{bradfield2007} for a survey of
the $\mu$-calculus including these results.

It seems that $L_μ$ strikes a good balance between expressivity and
complexity.  The computational complexity of the model-checking
problem, i.e., the problem of checking whether a formula $\phi\in
L_\mu$ is true at a node $v$ of a structure $M$ (in this paper we use
the term \emph{structure} for \emph{transition systems} or
\emph{Kripke structures}) is of particular interest, especially in the
field of formal verification.  The problem is polynomial-time
reducible to the problem of determining the winner of a parity game, a
certain kind of 2-player game played on directed graphs, and most
approaches for analyzing the complexity of $L_\mu$ model-checking are
based on parity games.

The problem of determining the winner of a parity game is in
$\text{NP} \cap \text{coNP}$, and in fact it is even in $\text{UP}
\cap \text{coUP}$~\cite{jurdzinski1998}. Despite 30 years of research,
the question whether parity games can be decided in polynomial time is
a long-standing open problem in the theory of logics for verification.

As a precise analysis of the classical complexity of $L_\mu$
model-checking remains elusive, we study the problem within the
framework of parameterized complexity
theory~\cite{DowneyF98,FlumGro06}. In particular, we aim at algorithms
verifying whether a formula $\phi$ is true at a node $v$ in a
structure $M$ in time $f(\phi)\cdot \abs{M}^c$, where $f$ is a
computable function from formulas into the positive integers and $c$
is a constant independent of $\phi$. Computational problems that can
be solved in this way, i.e., in time $f(k)\cdot n^c$, where $n$ is the
size of input and $k$ is a \emph{parameter} of the input, a natural
number such as length or quantifier-depth of a formula, are called
\emph{fixed-parameter tractable (fpt)} and the class of all fpt
problems is denoted FPT.

The parameterized complexity of logics such as monadic second-order
logic (MSO) or first-order logic (FO) has been well studied in the
literature, especially in the context of algorithmic
meta-theorems. See e.g.~\cite{GroheK11} for a recent survey. However,
not much is known about the parameterized complexity of $L_\mu$. As
every $L_\mu$-formula can be translated into an equivalent MSO
formula, fpt results for MSO immediately imply fpt results for
$L_\mu$.  As a consequence, $L_\mu$ is fpt on classes of structures of
bounded clique-width~\cite{CourcelleMakRot00},
bi-rank-width~\cite{corr/abs-0709-1433} or
tree-width~\cite{Courcelle90}.  However, besides these results that
follow from embedding into MSO, we are not aware of any other
tractable cases.

On the other hand, we know more about solving parity games on
restricted classes.  One of the first results in this direction was by
Jan Obdrz{\'a}lek~\cite{Obdrzalek03}, who showed that parity games of
bounded tree-width can be solved in polynomial time.  This result was
later extended to bounded clique-width~\cite{Obdrzalek07}.  Since
parity games are directed graphs, it is natural to look for graph
measures taking the direction of edges into account.  Such measures
include directed path-width~\cite{Barat06},
DAG-width~\cite{BerwangerDHKO12}, Kelly-width~\cite{hunter2008},
directed tree-width~\cite{johnson2001} and entanglement~\cite{BGKR12}.
Classes of parity games for which any of these measures is bounded can
be solved in polynomial time
(see~\cite{BerwangerDHKO12,hunter2008,BGKR12}), with the exception of
directed tree-width. Solving parity games in polynomial time on
directed tree-width is still an open problem.

A class of digraphs where the DAG- or Kelly-width is bounded also has
bounded directed tree-width.  DAG-width and Kelly-width are as yet
uncomparable concepts.  However, any class of digraphs of bounded
directed path-width has bounded Kelly- and DAG-width, which implies
polynomial time solvability of parity games of bounded directed
path-width by the results cited above.

\smallskip\noindent\textbf{Our contributions. }  The aim of this paper
is to develop the logical and algorithmic tools for proving
fixed-parameter tractability of $L_\mu$-model-checking on special
classes of structures such as classes of bounded Kelly-width.

Such classes already contain natural and interesting examples of
transition systems. However, we see our work also as a first step in a
more general program of showing that $L_\mu$-model-checking is fpt in
general. For this, it is easily seen that it suffices to solve the
problem on planar structures. We therefore aim, as a next step, to
show that it is fpt on classes of planar structures of bounded
directed tree-width. A general duality theorem \cite{KawarabayashiK14}
states that if the directed tree-width is high, then the structure
contains a grid-like substructure. In the planar case, this yields a
natural decomposition of the structure into smaller substructures
which can possibly be exploited for solving $L_\mu$-model-checking for
structures of very high directed tree-width. The techniques we develop
in this paper are a first step towards this goal and we believe that
they will prove useful for classes of structures beyond bounded
Kelly-width or bounded DAG-width.

Furthermore, besides the algorithmic applications, we believe that the
decomposition theorems we establish below may be of independent
interest.

\smallskip

\noindent\textit{Main contributions to logic of this paper. }
An important logical tool in the analysis of the parameterized
complexity of model checking for FO or MSO are \emph{decomposition
  theorems}, also referred to as Feferman-Vaught style theorems
(see~\cite{makowsky2004} for a comprehensive survey).  Whereas for FO
and MSO a range of such theorems are known, much less seems to be
available for $L_\mu$.  In this paper we prove a general decomposition
theorem for $L_μ$ that allows us to compute the formulas true at a
node in a structure from the formulas true at the nodes in some
induced substructures.  Our theorem is similar in spirit to the
theorem by Feferman and Vaught on disjoint unions~\cite{feferman1959}.
As far as we are aware, no such theorem was known for $L_μ$ prior to
our work.

The first step for such a theorem is finding a useful notion for the
``depth'' of a formula, so that up to equivalence there are only
finitely many formulas up to a given depth, and that the types of the
nodes in the full structure can be computed from the types of the
nodes in some induced substructures.  We propose the notion of
\emph{$\mu$-depth} that satisfies both constraints.

In this paper we study the construction of a structure $M$ from two
structures $M_1$ and $M_2$ where $M$ is defined as the union of $M_1$
and $M_2$ plus an arbitrary set of edges from $M_1$ to $M_2$. We call
the pair $(M_1, M_2)$ a \emph{directed separation} of $M$ and refer to
the intersection $M_1\cap M_2$ as the \emph{interface}. See
\cref{def:dir-separation} for details. Let $(M_1, M_2)$ and $(M_1,
M'_2)$ be two directed separations with interface $X$ as defined
above. Note that both have the same left-hand side $M_1$.  For a given
$\mu$-depth $\delta$, we define a notion of $\delta$-equivalence on
these separations. The main ingredient of $\delta$-equivalence is that
$M_2$ and $M_2'$ realize the same $L_\mu$-types up to $\mu$-depth
$\delta$, when the interface nodes are indicated with special
predicates. See \cref{def:L_equivalent} for details.

\begin{theorem*}[\cref{thm:types_suffice-delta}]
  Let $\delta$ be a $\mu$-depth, and let $M=(M_1,M_2)$,
  $M'=(M_1,M_2')$ be $\delta$-equivalent directed separations. Then
  for every node in $M_1$, the set of formulas of depth $\delta$ that
  it satisfies is the same in $M$ and in $M'$.
\end{theorem*}

The theorem, apart from its purely logical appeal, also has
applications for $L_\mu$-model checking.  The notion of equivalent
structures $(M_1, M_2)$ and $(M_1, M_2')$ can also be read in the way
that, given a huge structure $(M_1, M_2)$, we can replace $M_2$ by a
much smaller structure as long as it realizes the same types up to a
certain depth. This will be the main tool in our algorithmic
applications.

\smallskip

\noindent\textit{Applications to $L_\mu$-model checking. }
Based on our decomposition theorems above, we show that $L_\mu$-model
checking is fpt on classes of structures of bounded Kelly-width or
bounded DAG-width, provided a decomposition is given as part of the
input.

\smallskip

\noindent\textbf{Relation to other work. }
A natural idea for solving $L_\mu$-model-checking on a class $\CCC$ of
structures of bounded Kelly-width would be to reduce the problem to
parity games and apply the polynomial-time algorithms for solving
parity games of bounded Kelly-width.  However, the degree of the
polynomial-time algorithms for parity games
in~\cite{BerwangerDHKO12,hunter2008} depends on the upper bound for
the Kelly- or DAG-width of the games considered.  By combining a
structure of Kelly-width $k$ and a formula $\phi$ into a parity game,
the resulting game may have Kelly-width in the order of $k\cdot
\abs{\phi}$.  Hence, by translating into parity games we would not
obtain fpt algorithms.

The polynomial-time algorithms for parity games developed
in~\cite{Obdrzalek03,BerwangerDHKO12,hunter2008} all rely in some way
on the concept of \emph{borders}, \emph{strategy profiles} and
\emph{interfaces} developed first in~\cite{Obdrzalek03}, the paper on
parity games on bounded tree-width.  Our results also make crucial use
of these concepts. The main technical challenge we need to solve is
that for our decomposition theorems we need these profiles to be
definable in the $\mu$-calculus in a uniform way, which was not
necessary in the algorithmic papers on parity games.

%%% Local Variables: 
%%% mode: latex
%%% TeX-master: "types"
%%% End: 

\section{A Decomposition Theorem for \texorpdfstring{$L_\mu$}{L\_mu}}
\label{sec:definitions}

In this section we present the statement of our decomposition theorem
for the $\mu$-calculus.  We propose a notion of depth for formulas of
the $\mu$-calculus and then state \cref{thm:types_suffice-delta},
which says that this notion of depth is exactly what we want for our
decompositions.

\subsection{Syntax and Semantics of the Modal
  \texorpdfstring{μ}{mu}-Calculus}
\label{sec:syntax_semantics_mu}

We use the usual definition of the modal μ-calculus $L_μ$, see for
example in the comprehensive survey~\cite{bradfield2007}.  Let us
briefly review these definitions.

Let $\Var$ be an infinite set of \emph{fixpoint variables} and
$\sigma$ be a \emph{signature}, that is a set of \emph{atomic
  propositions}.  We define the formulas of $L_μ[\sigma]$ recursively.
\begin{itemize}
\item $⊤,⊥ ∈ L_μ[\sigma]$.
\item For all $P ∈ \sigma$, $P, \neg P ∈ L_μ[\sigma]$.
\item For all $X ∈ \Var$, $X ∈ L_μ[\sigma]$.
\item For all formulas $φ, ψ$, $(φ ∧ ψ), (φ ∨ ψ) ∈ L_μ[\sigma]$.
\item For all formulas $φ$, $(□φ), (◊φ) ∈ L_μ[\sigma]$.
\item For all formulas $φ$ and $X ∈ \Var$, $(μX.φ), (νX.φ) ∈
  L_μ[\sigma]$.
\end{itemize}
We omit brackets and $\sigma$ if there is no confusion.  With this
definition all formulas are in \emph{negation normal form}, that is,
negations may only appear in front of propositions.  There are more
general definitions of $L_μ$ with regard to negation, but every such
formula is equivalent to a formula in negation normal form.

The semantics of the $\mu$-calculus is defined on $\sigma$-structures,
also known as labelled transition systems or Kripke structures.
\begin{definition}
  A \emph{$\sigma$-structure} $M$ over a signature $\sigma$ is a
  directed graph together with a distinguished set of vertices $X(M)$
  for every $X \in \sigma$.

  We often use $M,v$, that is, a structure $M$ together with a
  distinguished node $v ∈ V(M)$.
\end{definition}

We use standard notation from model theory and graph theory.  In
particular, for $X ⊆ V(M)$, we write $M[X]$ for the substructure
induced by $X$.

The notion of a fixpoint variable being free or bound in a formula is
defined the standard way.  We write $\free(φ)$ for the set of free
fixpoint variables of $φ$.  Let $\varphi$ be a formula of
$L_\mu[\sigma]$.  To evaluate this formula, we use a $\tau$-structure
$M$ together with a distinguished vertex $v$ for some $\tau \supseteq
\sigma$.  The semantics relation $M,v \models \varphi$ is defined by
induction on $\varphi$ as follows.
\begin{itemize}
\item $M,v ⊧ ⊤$ and $M,v ⊭ ⊥$.
\item $M,v ⊧ P$ iff $v ∈ P(M)$ and $M,v ⊧ \neg P$ iff $v ∉ P(M)$ for
  $P ∈ \tau$.
\item $M,v ⊧ φ ∨ ψ$ iff $M,v ⊧ φ$ or $M,v ⊧ ψ$.
\item $M,v ⊧ φ ∧ ψ$ iff $M,v ⊧ φ$ and $M,v ⊧ ψ$.
\item $M,v ⊧ ◊φ$ iff there is $(v,w) ∈ E(M)$ with $M,w ⊧ φ$.
\item $M,v ⊧ □φ$ iff for all $(v,w) ∈ E(M)$, $M,w ⊧ φ$.
\item $M,v ⊧ μX.φ$ iff \[ v ∈ \bigcap\set[S ⊆ V(M)]{S \supseteq
    \set[v]{M[X/S],v ⊧ φ}}. \]
\item $M,v ⊧ νX.φ$ iff \[ v ∈ \bigcup\set[S ⊆ V(M)]{S \subseteq
    \set[v]{M[X/S],v ⊧ φ}}. \]
\end{itemize}
where $M[X/S]$ is the $\tau \cup \set{X}$-structure defined as $M$
extended by the interpretation $X(M[X/S]) \coloneq S$.

\subsection{A Notion of Formula Depth for the
  \texorpdfstring{$\mu$}{mu}-Calculus}

\begin{definition}\label{def:consistent_fixpoint_labelling}
  Let $\seq X = (X_1,\ldots,X_n)$ be a finite sequence of fixpoint
  variables.  A formula $\varphi ∈ L_μ$ is called \emph{consistent}
  with $\seq X$ if all fixpoint variables of $\varphi$ (free and
  bound) are in the sequence, and in every subformula $\psi$ of
  $\varphi$ that binds a fixpoint variable $X_i$, only the variables
  $X_1,\ldots,X_{i}$ can appear freely in $\psi$.
\end{definition}

\begin{figure}[ht]
  \centering
  \includegraphics[scale=.6]{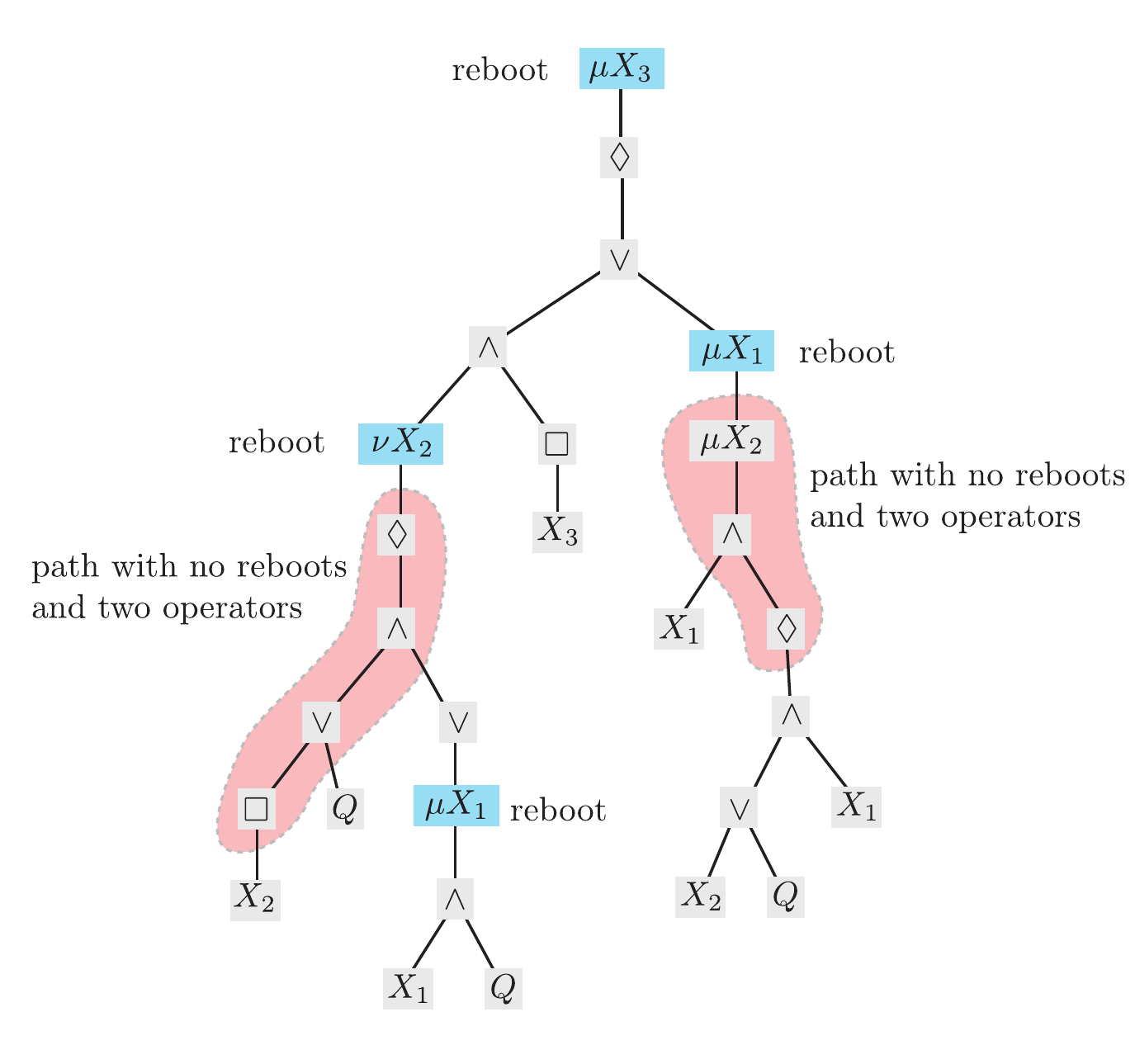}
  \caption{An example for reboots}
\label{fig:reboot}
\end{figure}

A node $x$ in the syntax tree of a formula that is consistent with
$\seq X$ is called a \emph{reboot} if the subformula in the node binds
a fixpoint variable $X_i$ such that no ancestor of $x$ binds any of
the fixpoint variables $\{X_1,\ldots,X_i\}$. The \emph{$\seq X$-depth}
of a formula is the biggest number of occurrences of operators from
the set $□, ◊, \mu, \nu$ that can be found on a path in the syntax
tree that does not visit reboot nodes. The $\seq X$-depth is undefined
if the formula is not consistent with $\seq X$.  \Cref{fig:reboot}
shows a formula which has $(X_1,X_2,X_3)$-depth~2.

The definition is designed so that $\mu X.φ$ and $φ[X/\mu X.φ]$ have
the same $\seq X$-depth.

For a set $L \subseteq L_\mu$, define the \emph{$L$-type} of a vertex
in a structure to be the set of formulas from $L$ that are true at the
vertex.  A \emph{$\mu$-depth} is a pair $\delta=(\seq X,d)$ where
$\seq X$ is a sequence of fixpoint variables and $d$ is a natural
number.  A formula is called \emph{consistent with $\delta$} if it is
consistent with $\seq X$ and its $\seq X$-depth is at most $d$. The
$\delta$-type of a vertex in a structure is its $L$-type, with $L$
being the set of all formulas consistent with $\delta$.  This
information is finite thanks to the following lemma.

\begin{lemma}
	For every $\mu$-depth $\delta$ and finite set of propositional
  variables, up to logical equivalence there are finitely many
  formulas in these propositional variables that are consistent with
  $\delta$.
\end{lemma}
\begin{proof}
  Define the \emph{standard depth} of a formula to be the biggest
  number of operators from $□, ◊, \mu, \nu$ on any path in the syntax
  tree. It is not difficult to see that, when the set of propositional
  variables is fixed, up to logical equivalence there are finitely
  many formulas of given standard depth.  If $\delta = (\seq X,d)$
  with $\seq X = (X_1,\ldots,X_n)$, then a formula with $\seq X$-depth
  $d$ has standard depth at most $d \cdot (n+1)$, so the result
  follows.
\end{proof}

Although the set in the statement of the above lemma is finite, its
size is non-elementary with respect to $\delta$. % TODO: why?

\subsection{Decompositions of Directed Separations}
As promised in the introduction, we will prove a decomposition theorem
for the union of two structures with a small intersection and some
additional edges all going in the same direction.  To formalize this,
we introduce \emph{directed separations}.

\begin{definition}\label{def:dir-separation}
  Let $M$ be a $\sigma$-structure.  A pair $(M_1,M_2)$ of induced
  substructures is a \emph{directed $\sigma$-separation of $M$ with
    interface $\seq X = (x_1,\ldots,x_k)$} if
  \begin{itemize}
  \item $V(M) = V(M_1) \cup V(M_2)$,
  \item $X = \set{x_1,\ldots,x_k} = V(M_1) \cap V(M_2)$,
  \item and there are no edges from $M_2 \setminus X$ to $M_1
    \setminus X$.
  \end{itemize}
\end{definition}
Abusing notation, we write $M = (M_1,M_2)$ to denote that $(M_1,M_2)$
is a directed separation of $M$, and notationally we consider
$(M_1,M_2)$ to be interchangeable with $M$.

For some $k$, let $\seq{P} = P_1,\ldots,P_k$ be a sequence of fresh
proposition symbols.  For a $\sigma$-structure $M$ and a $k$-tuple
$\seq{X} = (x_1,\ldots,x_k) ∈ V(M)^k$, we define $\colored{M}{\seq
  P}{\seq X}$ to be the $\sigma \cup P$-structure based on $M$ such
that $P_i$ is true only at the node $x_i$.  If the sequence $\seq P =
P_1,\ldots,P_k$ is longer than $\seq X = (x_1,\ldots,x_l)$, then
$\colored{M}{\seq P}{\seq X}$ is defined the same except that $P_i$ is
always false for $i > l$.

\begin{definition}\label{def:L_equivalent}
  Let $(M_1,M_2)$, $(M_1,M_2')$ be two directed separations with the
  same interface $\seq X$.  Let $\seq P$ be a set of $\abs{X}$~many
  proposition symbols and let $L ⊆ L_μ[\sigma \cup P]$.

  We call $(M_1,M_2)$, $(M_1,M_2')$ \emph{$L$-equivalent} if
  \begin{itemize}
  \item for every vertex in $X$, its $L$-type is the same in $\colored
    {M_2}{\seq P}{\seq X}$ and $\colored{M_2'}{\seq P}{\seq X}$,
    respectively, and
  \item for every edge $(v,w)$ in $(M_1,M_2)$ with $v ∈ M_1, w ∈ M_2$
    there is an edge $(v,w')$ in $(M_1,M_2')$ with $w ∈ M_2'$ such
    that $w$ and $w'$ have the same $L$-types in $\colored{M_2}{\seq
      P}{\seq X}$ and $\colored{M_2'}{\seq P}{\seq X}$, respectively,
    and vice versa.
  \end{itemize}
\end{definition}

If $\delta$ is a $\mu$-depth, we say that two directed separations are
\emph{$\delta$-equivalent} if they are $L$-equivalent with $L$ being
all formulas consistent with $\delta$.  Let us state our main theorem.

\begin{theorem}\label{thm:types_suffice-delta}
	Let $\delta$ be a $\mu$-depth, and let $M=(M_1,M_2)$,
  $M'=(M_1,M_2')$ be $\delta$-equivalent directed separations. Then
  for every node in $M_1$, its $\delta$-type is the same in $M$ and
  $M'$.
\end{theorem}

In fact, we will prove a more general version of
\cref{thm:types_suffice-delta}, without limiting us to $\mu$-depth.
It turns out that there exists a suitable closure operator $\CL_P:
2^{L_μ} \to 2^{L_μ}$ that maps finite sets to finite sets such that
the main theorem holds for $\CL_P(L)$-equivalent directed separations.
In particular, we can choose $L = \set{φ}$ if we are only interested
in the model checking problem for a fixed formula $φ$ consistent with
$\delta$.  Then $\CL_P(\set{φ})$ will be significantly smaller than
the set of all $\delta$-consistent formulas.

%%% Local Variables:
%%% TeX-master: "types"
%%% End:

\section{Proof of the Decomposition Theorem}
\label{sec:main_theorem}

\begin{definition}
  For $φ ∈ L_μ$, let $\sub(φ)$ be the set of all indexed subformulas
  without formulas of the form $X$ for fixpoint variables $X$.
  That is,
  \begin{align*}
    \sub(φ) &\coloneq \{(ψ,i) \mid \text{$ψ$ is a subformula of $φ$ at
      position $i$}\\
    &\qquad\text{in the string φ and ψ is not a variable} \}.
  \end{align*}
  Let $\sub^+(φ) = \sub(φ) \setminus \set{(φ,0)}$ be the set of proper
  subformulas.

  For an occurrence of a fixpoint variable $X$ in a formula $φ$, its
  \emph{definition in φ} is the enclosing fixpoint $(μX.ψ,i) ∈
  \sub(φ)$ (or $(νX.ψ,i) ∈ \sub(φ)$) where this occurrence of $X$ is
  quantified.  For a formula $(\psi,i) ∈ \sub(φ)$, let
  $\closure_φ(\psi,i) = (\psi',i)$ be such that $ψ'$ is the formula
  $ψ$ with all free variables replaced by their definitions until
  there are no more free variables.

  Define $\CL(φ) \coloneq \set{ \closure_φ(ψ,i) \mid (ψ,i) ∈ \sub(φ)
  }$ and $\CL^+(φ) = \CL(φ) \setminus \set{(φ,0)}$.
\end{definition}
We will often not distinguish formulas in $\sub(φ)$ and $\CL(φ)$ and
instead identify them via the obvious bijection that preserves the
second component.

We will usually write $ψ ∈ \sub(φ)$ instead of $(ψ,i) ∈ \sub(φ)$ if
there is no confusion.  We only need the index $i$ in order to
distinguish identically looking subformulas.

Even though two subformulas may look identical, they could be in the
scope of different fixpoint operators.  A few paragraphs below we will
introduce a closure operation called $\PT_P$ that modifies different
subformulas in different ways in order to distinguish between these
cases.  For this reason we need to keep track of the positions of the
subformulas.  In the rest of the paper, whenever we mention an element
of $\sub(φ)$ or $\CL(φ)$, the reader should assume that it also
contains the position of the subformula in $φ$.

\begin{lemma}
  For all $φ ∈ L_μ$, the set \[ \{ ψ \mid (ψ,i) ∈ \CL(φ) \text{ for
    some $i$}\} \] is equal to the usual definition of the
  Fischer-Ladner closure of $φ$ (see e.g.,
  \cite[Definition~4.1]{streett1989}).
\end{lemma}

For a set of formulas $L ⊆ L_μ$, define $\CL(L) \coloneq \{\psi \mid
\text{$φ ∈ L$, $(\psi,i) ∈ \CL(φ)$}\}$.  In this set we do not need
the index $i$, different from $\CL(φ)$.

Let $P = \set{P_1,\ldots,P_k}$ be a set of proposition symbols
disjoint from $\sigma$.

For a formula $φ ∈ L_μ[\sigma \cup P]$ and $φ' ∈ L_μ[\sigma \cup P]$,
we call $φ'$ a \emph{priority tracking variant of φ} if $φ'$ is
syntactically derived from $φ$ by applying the following operation for
each subformula $ψ$ of the form $ψ = ◊\chi$ or $ψ = □\chi$.
\begin{enumerate}
\item If $ψ = ◊\chi$, then pick a set $Q ⊆ P$ and replace the
  subformula $ψ$ by \[ \biggl(\Bigl(\bigvee_{R ∈ Q} R\Bigr) ∨
  ◊\chi\biggr). \]
\item If $ψ = □\chi$, then pick a set $Q ⊆ P$ and replace the
  subformula $ψ$ by \[ \biggl(\Bigl(\bigwedge_{R ∈ Q} \neg R\Bigr) ∧
  □\chi\biggr). \]
\end{enumerate}
We denote the set of all priority tracking variants of $φ$ with
respect to $P$ by $\PT_P(φ)$.  Note that $\PT_P(φ)$ is finite because
$φ$ has a finite number of subformulas and $P$ is a finite set.
Similar to $\CL$, we define $\PT_P(L)$ for sets of formulas $L ⊆
L_μ[\sigma \cup P]$ as $\PT_P(L) \coloneq \bigcup_{φ ∈ L} \PT_P(φ)$.

\begin{definition}
  Let $\CL_P(L) \coloneq \PT_P(\CL(L))$.
\end{definition}

\begin{lemma}\label{lem:clx_idempotent}
  $\CL_P(\CL_P(L)) = \CL_P(L)$ for all $L ⊆ L_μ[\sigma \cup P]$.
\end{lemma}
\begin{proof}
  By definition, $\CL_P(L)$ is closed under $\PT_P$. Hence, it is
  enough to show $\CL(\CL_P(L)) = \CL_P(L)$.  Let $μX.φ ∈ \CL_P(L)$.
  We want to show that $φ[X/μX.φ] ∈ \CL_P(L)$, where $φ[X/μX.φ]$ is φ
  with all free occurrences of $X$ replaced by $μX.φ$.  By definition
  of $\PT_P$, there is a $μX.φ' ∈ \CL(L)$ such that $φ$ is priority
  tracking variant of $φ'$.  Because $\CL(L)$ is essentially equal to
  the Fischer-Ladner closure of $L$, we have $φ'[X/μX.φ'] ∈ \CL(L)$.
  Since $φ$ is a priority tracking variant of $φ'$, the formula
  $φ[X/μX.φ]$ is a priority tracking variant of $φ'[X/μX.φ']$, hence
  $φ[X/μX.φ] ∈ \PT_P(\CL(L))$.

  The other cases are similar.
\end{proof}

\begin{definition}\label{def:L_types}
  For a structure $M,v$, a $k$-tupel $\seq X ∈ V(M)^k$ and a set $L ⊆
  L_μ[\sigma \cup P]$ where $\seq P$ is sequence of at least $k$~many
  proposition symbols, we define the \emph{$(L,\seq P)$-type of $v$ in
    $M,\seq X$} as
  \[ \ctypeWithIndex{L}{\seq P}{M}{v}{\seq X} \coloneq \set[{φ ∈
    \CL_P(L)}]{\colored{M}{\seq P}{\seq X},v ⊧ φ}. \]
  % We will omit the index $L$ if it is clear from the context.

  We also define the set of $(L,\seq P)$-types realized in a
  structure, \[ \structureTypeWithIndex{L}{\seq P}{M}{\seq X} \coloneq
  \set[\ctypeWithIndex{L}{\seq P}{M}{v}{\seq X}]{v ∈ V(M)}. \]
  Finally, let $ \allTypesWithIndex{\seq P}{L} \coloneq 2^{\CL_P(L)}$
  be the set of all candidates for $(L,\seq P)$-types.
\end{definition}

Using the new terminology, let us restate
\cref{thm:types_suffice-delta} in these more general terms.

\begin{theorem}\label{thm:types_suffice}
  Let $\seq P$ be a sequence of proposition symbols disjoint from
  $\sigma$, $L ⊆ L_μ[\sigma \cup P]$ and let $(M_1,M_2)$, $(M_1,M_2')$
  be $\CL_P(L)$-equivalent directed $\sigma$-separations with
  interface $\seq X$.

  Then for all $v ∈ M_1$, we have \[ \ctypeWithIndex{L}{\seq
    P}{(M_1,M_2)}{v}{\seq X} = \ctypeWithIndex{L}{\seq
    P}{(M_1,M_2')}{v}{\seq X}. \]
\end{theorem}

It is not difficult to show that if all formulas in $L$ are consistent
with a $\mu$-depth $\delta$, then the same is true for $\CL_P(L)$
(this is \cref{lem:CL_consistent}).  Therefore, $\delta$-equivalence
implies $\CL_P(L)$-equivalence, and thus
\cref{thm:types_suffice-delta} follows from \cref{thm:types_suffice}.
We will also use a different and slightly stronger way of stating
\cref{thm:types_suffice}, stated below.
\begin{theorem}\label{thm:computing_types}
  Let $\seq P$, $\seq Q$ be sequences of proposition symbols such that
  $\sigma \cap P = \sigma \cap Q = P \cap Q = \emptyset$.

  Let $L ⊆ L_μ[\sigma \cup P]$ and $M$ be a structure with a directed
  $\sigma$-separation $(M_1,M_2)$ with interface $\seq X$.  Let $\seq
  Y ∈ V(M_1)^{\abs{Q}}$ be a tuple.

  For all $v ∈ M_1$, the set $\ctypeWithIndex{L}{\seq Q}{M}{v}{\seq
    Y}$ depends only on
  \begin{itemize}
  \item $M_1$ and $\seq Q$ and
  \item $\set[{(x_i,\ctypeWithIndex{L}{\seq P}{M_2}{x_i}{\seq
        X})}]{x_i ∈ X}$ and
  \item $\set{ (v,\ctypeWithIndex{L}{\seq P}{M_2}{w}{\seq X}) \mid
      \text{$(v,w) ∈ E(M) \cap (M_1 \times M_2)$} }$.
  \end{itemize}
  Provided $L$ is finite, $\ctypeWithIndex{L}{\seq Q}{M}{v}{\seq Y}$
  can be computed from these sets.

  Furthermore, for every $w ∈ M_2$, the set $\ctypeWithIndex{L}{\seq
    Q}{M}{w}{\seq Y}$ depends only on the above sets and on
  $\ctypeWithIndex{L}{\seq P}{M_2}{w}{\seq X}$ and can be computed
  from these sets if $L$ is finite.
\end{theorem}

\subsection{Parity Games}

To prove the decomposition theorems, we want to use the model checking
game of the modal μ-calculus.  Instead of replacing a substructure by
a different substructure preserving the types in the whole structure,
we replace a subgame by a different subgame preserving the winner in
the whole game.

For this, we first need parity games, strategies and the model
checking game.  These are all well-known concepts in the literature,
see for example~\cite{automata2002}.  We briefly review the key
concepts.

The winner of a parity game from a given node is always determined.
However, in order to replace subgames by different subgames preserving
the winner in the whole game, we need a more subtle analysis of the
subgame than just its winner.

We call the intersection between a subgame and the rest of the game
its \emph{interface}.  For the more subtle analysis, we look at
\emph{partial} strategies, which may be undefined on some nodes of the
interface.  If a partial strategy is undefined on some node, the
player indicates that she would like to leave the subgame.  These
strategies can be partially ordered by their \emph{profiles}, that is,
the set of interface nodes that are possibly reachable by \podd,
together with the worst priority that \podd can enforce.

All this culminates in a proof that the feasibility of profiles of
strategies is in fact definable in $L_μ$.  The formulas that define
profiles in a partial model checking game of $φ$ will all be in
$\CL_P(\set{φ})$, so this proves that $\ctypeWithIndex{\set{φ}}{\seq
  P}{M}{v}{\seq X}$ determines the set of possible profiles, which we
will use to define a specific parity game.

\newcommand{\priority}[2]{\ensuremath{\overset{#1}{#2}}}

Let $\seq Z = (Z_1,\ldots,Z_n)$ be a finite sequence of fixpoint
variables.  Recall the definition of a formula consistent with $\seq
Z$ (\cref{def:consistent_fixpoint_labelling} on
\cpageref{def:consistent_fixpoint_labelling}).  We strengthen this
definition in the sense that every $Z_i$ is either bound only in
$μ$-subformulas or only in $ν$-subformulas.  Let $(p_1,\ldots,p_n)$ be
a strictly increasing sequence of natural numbers such that $p_i$ is
odd if and only if $Z_i$ is only bound in $μ$-subformulas.

Let $φ ∈ L_μ$ be consistent with $\seq Z$ and $\mu Z_i.\psi ∈
\sub(φ)$.  We write $\priority{p_i}{μ} Z_i. \psi$ to indicate that
$Z_i$ gets the priority $p_i$ in the model checking game that we will
define shortly (similarly for $ν$).  We call a formula with numbers
over their fixpoint operators an \emph{annotated} formula.  In this
section it does not affect the results if the sequences are infinite.

From now on, let us fix a sequence $\seq Z$ and a corresponding
priority sequence $(p_1,\ldots,p_n)$.  All formulas in the rest of
this section should be consistent with $\seq Z$ and annotated with the
$p_i$, even if we do not mention this explicitly.  For example, a
formula $νY.◊(μX.νY.◊X ∨ ◊Y) ∨ ◊Y$ consistent with $(X,Y)$ under the
priority sequence $(1,2)$ would be labelled as
$\priority{2}{ν}Y.◊(\priority{1}{μ}X.\priority{2}{ν}Y.◊X ∨ ◊Y) ∨ ◊Y$.
Note that it cannot be labelled
$\priority{2}{ν}Y.◊(\priority{3}{μ}X.\priority{4}{ν}Y.◊X ∨ ◊Y) ∨ ◊Y$,
even though these priorities would work in the model-checking game.
However, they violate the sequence $(X,Y)$ and the priority sequence
$(1,2)$.

Note that the first formula is an element of
$\CL(\priority{1}{μ}X.\penalty0\priority{2}{ν}Y.◊X ∨ ◊Y)$.  This holds
true in general.

\begin{lemma}\label{lem:CL_consistent}
  Let $\seq Z = (Z_1,\ldots,Z_n)$, $φ ∈ L_μ$ be consistent with $\seq
  Z$ and $ψ ∈ \CL(φ)$.  Then $\psi$ is consistent with $\seq Z$.
\end{lemma}
\begin{proof}
  We prove this by structural induction.  If $φ = \chi_1 ∧ \chi_2$,
  then obviously $\chi_1$ and $\chi_2$ are consistent with $\seq Z$.
  The same is true for most other cases.  The only interesting cases
  are the fixpoints.  We only consider the case of $μ$ fixed points,
  the other cases follow analogously.

  Assume that $φ = μ Z_i.ψ$.  We need to show that $ψ[Z_i/φ]$ is
  consistent with $\seq Z$.  Because $φ$ is consistent with $\seq Z$,
  the only places where $ψ[X/φ]$ could become inconsistent is a
  subformula of the form $φ$ inside the scope of another fixpoint
  operator $μ Z_j.\chi$ with $j > i$ with $Z_j$ being free in $φ$.
  This is impossible because $φ$ does not have free variables.
\end{proof}

Now let us briefly review the definitions of parity games, strategies
and model checking games.

\begin{definition}\label{def:parity_game}
  A \emph{parity game} $P = (V,V_\even,E,\omega)$ is a directed graph
  $(V,E)$ with $V_\even ⊆ V$ and a function $\omega: V \to \N$ mapping
  nodes to \emph{priorities}.
\end{definition}

A parity game is played by two
players, \peven and \podd.  The game starts on a node $v$.  It is
\peven's turn if the current node is in $V_\even$, otherwise it is
\podd's turn.  In their turn, the players must choose an outgoing edge
and the endpoint becomes the current node for the next turn.  If a
player cannot make a move, he loses.  Otherwise, the game continues
indefinitely.

The set of nodes visited during an infinite play is an infinite path
$\overline{v} = v_1,v_2,\ldots$.  Let $p$ be the minimum priority that
occurs infinitely often on $\overline{v}$.  The path $\overline{v}$ is
\emph{winning} for \peven if and only if $p$ is even.

\begin{definition}
  Let $P$ be a parity game.  For a partial function $\pi: V(P)^+ \to
  V(P)$ on finite non-empty paths of nodes and a path
  $(v_1,\ldots,v_n) ∈ V(P)^+$, we say that $\overline{v}$ is
  \emph{$\pi$-conforming} if for all $i < n$ with $v_i ∈ V_◊(P)$, we
  have $(v_1,\ldots,v_i) ∈ \dom(\pi)$ and $(v_i,\pi((v_1,\ldots,v_i)))
  ∈ E(P)$.  An infinite path is $\pi$-conforming if all its initial
  segments are $\pi$-conforming.

  A \emph{strategy} for \peven for a game $(P,v_1)$ is a partial
  function $\pi: V(P)^+ \to V(P)$ with the following conditions.
  \begin{enumerate}
  \item For every $(v_1,\ldots,v_n) ∈ \dom(\pi)$, the sequence
    $(v_1,\ldots,\penalty0 v_n,\pi(v_1,\ldots,v_n))$ is a
    $\pi$-conforming path in $P$ with $v_n ∈ V_◊(P)$.
  \item \label{condition:strategy_is_total} For every $\pi$-conforming
    path $(v_1,\ldots,v_n)$, if $v_n ∈ V_◊(P)$, then $(v_1,\ldots,v_n)
    ∈ \dom(\pi)$.
  \end{enumerate}

  A strategy $\pi$ is \emph{winning for \peven} if every maximal
  $\pi$-conforming path is winning for \peven.  A game $P$ on a node
  $v$ is \emph{winning for \peven} if \peven has a winning strategy
  for $(P,v)$.
\end{definition}

\begin{definition}
  A strategy $\pi$ is \emph{positional} if $\pi(v_1,\ldots,v_n)$ only
  depends on $v_n$.
\end{definition}

When we talk about strategies and do not explicitly mention the
player, we assume that the strategy is meant for \peven.  The
following result is well-known (see e.g., \cite{zielonka1998}).

\begin{theorem}\label{thm:positional_strategies_exist}
  For every winning strategy $\pi$ for a game $(P,v)$, there exists a
  positional winning strategy $\rho$ for $(P,v)$.
\end{theorem}

Parity games are relevant because they are the model checking game for
the modal μ-calculus.

\begin{definition}\label{def:model_checking_game}
  For a $\sigma$-structure $M$ and a formula $φ ∈ L_μ[\sigma]$, let
  $\MCgame{M}{φ} = (V, V_◊, E, \omega)$ be the \emph{model checking
    game} defined as follows.
  \begin{align*}
    V(\MCgame{M}{φ}) &\coloneq M \times \CL(φ)
  \end{align*}
  There is a an edge from $(v,\ind{ψ})$ to $(w,\ind{\chi})$ if
  \begin{itemize}
  \item $v = w$ and $\ind{ψ} ∈ \set{\ind{\chi ∧ \chi'}, \ind{\chi ∨
        \chi'}, \ind{\chi' ∧ \chi}, \ind{\chi' ∨ \chi}}$ for some
    $\chi'$ or
  \item $v = w$, $\ind{\psi} ∈ \set{\ind{μX.\chi'}, \ind{νX.\chi'}}$
    and $\ind{\chi} = \closure_φ(\ind{\chi'})$ or
  \item $(v,w) ∈ E(M)$ and $\ind{ψ} ∈ \set{\ind{◊ψ}, \ind{□ψ}}$.
  \end{itemize}
  A node $(v,\ind{ψ})$ is a □-node if either
  \begin{itemize}
  \item $ψ = P$ and $M,v ⊧ P$ or
  \item $ψ ∈ \set{\chi ∧ \chi', □\chi}$ for some $\chi, \chi'$.
  \end{itemize}
  A node $(v,\ind{ψ})$ has the priority
  \begin{align*}
    \omega(v,\ind{ψ}) &\coloneq
    \begin{cases}
      p & \text{if $\ind{ψ} = \ind{\priority{p}{μ}X.\chi}$ or $\ind{ψ} = \ind{\priority{p}{ν}X.\chi}$ for some $\chi$}\\
      p' & \text{otherwise},
    \end{cases}
  \end{align*}
  where $p'$ is the maximum priority.
\end{definition}

It is easy to show that this definition gives a well-defined model
checking game (see e.g., \cite{automata2002/zappe}).

\subsection{Profiles and Types}
\label{sec:profiles}

In the previous section we considered parity games, (positional)
strategies and the model checking game.  We now generalize these
definition to partial games and partial strategies.  This is necessary
so we can analyze the effect of replacing a subgame by a different,
but in some sense similar subgame.

\begin{definition}
  A \emph{partial parity game} is a parity game $P$ with a subset $U ⊆
  V(P)$ called the \emph{interface}.
\end{definition}

The game is played the same way as a parity game, except that upon
reaching an interface \even-node, \peven may choose to end the play
and win immediately.  Therefore, a \emph{partial strategy} for \peven
is defined the same way as in a non-partial parity game, except that
the partial strategy may be undefined on plays that end in an
interface \even-node.

\begin{definition}
  Let $P$ be a partial parity game.  A \emph{partial strategy} for
  \peven for a game $(P,v_1)$ is a partial function $\pi: V(P)^+ \to
  V(P)$ with the following conditions.
  \begin{enumerate}
  \item For every $(v_1,\ldots,v_n) ∈ \dom(\pi)$, the sequence
    $(v_1,\ldots,\penalty0 v_n,\pi(v_1,\ldots,v_n))$ is a
    $\pi$-conforming path in $P$ with $v_n ∈ V_◊(P)$.
  \item For every $\pi$-conforming path $(v_1,\ldots,v_n)$, if $v_n ∈
    V_◊(P)$ and $v_n ∉ U$, then $(v_1,\ldots,v_n) ∈ \dom(\pi)$.
  \end{enumerate}
\end{definition}

A partial strategy $\pi$ is called \emph{winning} if for every
strategy of the opponent, the resulting play either visits an
interface node where $\pi$ is undefined or satisfies the parity
condition.  Formally, we define this as follows.

\begin{definition}
  Let $(P,v_1)$ be a partial parity game with interface $U$ and $\pi$
  be a partial strategy.  Let $P'$ be the game constructed from $P$ by
  adding a □-node called $\top$ and an edge from every node in $V_◊
  \cap U$ to $\top$.  Then define $\pi'$ as an extension of $\pi$ such
  that on all $\pi$-conforming paths $(v_1,\ldots,v_n)$ with $v_n ∈
  V_◊ \cap U$, if $(v_1,\ldots,v_n) ∉ \dom(\pi)$, then
  $\pi'((v_1,\ldots,v_n)) = \top$.  Then $\pi'$ is a strategy on
  $(P',v_1)$.

  We say that $\pi$ is a partial \emph{winning} strategy from node
  $v_1$ iff $\pi'$ wins from node $v_1$ in the game $P'$.
\end{definition}

If we have a structure together with some subset of its nodes, we
consider the corresponding model checking games to be partial with
respect to these nodes.

\begin{definition}
  Let $φ ∈ L_μ[\sigma]$, $M$ be a $\sigma$-structure and $X ⊆ V(M)$.
  The game $\MCgameP{X}{M}{φ}$ is the partial parity game defined as
  $\MCgame{M}{φ}$ with interface $\{(v,ψ) ∈ X \times \CL(φ) \mid
  \text{$ψ$ starts with ◊ or □}\}$.  We will usually write
  $\MCgame{M}{φ}$ for this game if $X$ is clear from the context.
\end{definition}

We emphasize again that $\MCgameP{X}{M}{φ}$ and $\MCgameP{M}{φ}$ are
exactly the same game, only viewed from two different angles.

\begin{definition}
  Let $P$ be a partial parity game with interface $U$.  We define
  \begin{align*}
    &\strategytargets(P) \coloneq \set[(u,p)]{\text{$u∈U$, $p$ a priority of $P$}}\\
    &\strategyprofiles(P) \coloneq \{y ⊆ \strategytargets(P) \mid {}
    \text{for all $u$}\\
    &\hspace{2.4cm}\text{there is at most one $p$ with $(u,p) ∈ y$} \}.
  \end{align*}
\end{definition}
\begin{definition}
  Let $\sqsubseteq$ be the \emph{reward ordering} on priorities.  That
  is, $p \sqsubseteq p'$ if $p$ is better for \peven than $p'$.
  Formally, $p \sqsubseteq p'$ is true if and only if
  \begin{itemize}
  \item $p$ is even and $p'$ is odd or
  \item both $p$ and $p'$ are even and $p \leq p'$ or
  \item both $p$ and $p'$ are odd and $p \geq p'$.
  \end{itemize}
\end{definition}
\begin{definition}
  Let $P$ be a partial parity game with interface $U$, $v_1 ∈ V(P)$
  and let $\pi$ be a partial winning strategy for $(P,v_1)$.  We
  define
  \begin{align*}
    \strategypreprofile(\pi,v_1) &\coloneq \{ (v_n,\min_{1 \leq i \leq n} \omega(v_i)) \mid \\
    &\hspace{7.5mm}\text{$n>1$, $(v_1,\ldots,v_n)$ is a path with}\\
    &\hspace{7.5mm}\text{$v_n ∈ U$ and $(v_1,\ldots,v_n) ∉ \dom(\pi)$} \}\\
    \strategyprofile(\pi,v_1) &\coloneq \{ (u,p) \mid \text{$p$ is
      $\sqsubseteq$-maximal such that}\\
    &\hspace{7.5mm}\text{$(u,p) ∈ \strategypreprofile(\pi,v_1)$} \}.
  \end{align*}
  The $\min$ is taken with respect to the usual ordering $\leq$.

  We say that a profile $y ∈ \strategyprofiles(P)$ is \emph{possible}
  on $(P,v_1)$ if there exists a $\pi$ such that $y =
  \strategyprofile(\pi,v_1)$.
\end{definition}

\begin{definition}
  Let $y, y' ∈ \strategyprofiles(P)$.  We say that $y$ \emph{is at
    least as good as} $y'$ iff for every $(u,p) ∈ y$, there is a
  $(u,p') ∈ y'$ with $p \sqsubseteq p'$.  We denote this as $y
  \sqsubseteq y'$.
\end{definition}

{ % open a group so that the following definitions stay local
\tikzstyle{every node}=[circle,draw=black,fill=black,thick,inner sep=2.5pt,minimum size=2mm]
\tikzstyle{every path}=[->,thick,line width=1pt]
\tikzstyle{intersection1}=[fill=white,thick]
\tikzstyle{intersection2}=[fill=black!30]
\tikzstyle{boxplayer}=[rectangle]
\begin{figure}
  \centering
  \begin{tikzpicture}
    \begin{scope}[yscale=0.5,xscale=1,xshift=-1.5cm]
      \node [boxplayer,label=above:$v$] (A) at (0,0) {};
      \node (B1) at (-0.5,-1) {};
      \node (B2) at (0.5,-1) {};
      \node (C) at (0,-2) {};
      \node [intersection1] (X1) at (-0.5,-3) {};
      \node [intersection2] (X2) at (0.5,-3) {};
      \draw (A) -> (B1);
      \draw (A) -- (B2);
      \draw (B1) -- (C);
      \draw (B2) -- (C);
      \draw (C) -- (X1);
      \draw (C) -- (X2);
      \node [draw=none,fill=none] at (0,-4) {$P_1$};
    \end{scope}
    \begin{scope}[yscale=0.5,xscale=1,xshift=1.5cm]
      \node [boxplayer,label=above:$v$] (A) at (0,0) {};
      \node (B1) at (-0.5,-1) {};
      \node (B2) at (0.5,-1) {};
      \node (C1) at (-0.5,-2) {};
      \node (C2) at (0.5,-2) {};
      \node [intersection1] (X1) at (-0.5,-3) {};
      \node [intersection2] (X2) at (0.5,-3) {};
      \draw (A) -> (B1);
      \draw (A) -- (B2);
      \draw (B1) -- (C1);
      \draw (B2) -- (C2);
      \draw (C1) -- (X1);
      \draw (C2) -- (X2);
      \node [draw=none,fill=none] at (0,-4) {$P_2$};
    \end{scope}
  \end{tikzpicture}
  \caption{}
  \label{fig:profile_example}
\end{figure}
\newcommand{\inodeA}%
{\begin{tikzpicture}\node [intersection1] at (0,0) {};\end{tikzpicture}}
\newcommand{\inodeB}%
{\begin{tikzpicture}\node [intersection2] at (0,0) {};\end{tikzpicture}}

As an example, consider the two parity games given in
\cref{fig:profile_example} with interface nodes $\inodeA$, $\inodeB$.
For simplicity, we assume that all nodes in these parity games have
priority~0.  Then the profile $\set{(\inodeA,0)}$ is possible on
$(P_1,v)$ but not on $(P_2,v)$.  On the other hand, the profile
$\set{(\inodeA,0),(\inodeB,0)}$ is possible on both $(P_1,v)$ and
$(P_2,v)$.  Note that on $(P_1,v)$, the last profile is only possible
with a non-positional strategy.  However, the need for a
non-positional strategy here is of course somewhat artificial because
\peven must deliberately avoid a decision where she could simply
make one.
} % end the group so that the above definitions stay local

As one might expect, every partial strategy can be converted into a
positional partial strategy at least as good as the original strategy.
\begin{lemma}
  \label{lem:positional_partial_strategies_exist}
  Let $P = (V,V_◊,E,\omega)$ be a partial parity game with interface
  $U$, $v ∈ V$ and $\pi$ be a partial strategy for $(P,v)$.  Then
  there exists a positional partial strategy $\rho$ such that
  $\strategyprofile(\rho,v) \sqsubseteq \strategyprofile(\pi,v)$.
\end{lemma}
\begin{proof}
  The proof is a reduction to the positional determinacy of
  (non-partial) parity games.

  We define a game $P' = (V',V_◊',E',\omega')$ based on $P$ and use
  \cref{thm:positional_strategies_exist}.  Let
  \begin{align*}
    V' &\coloneq V \cup \set[v_p]{\text{$p$ is a priority of $P$}} \cup \set{v_⊥}\\
    V_◊' &\coloneq V_◊ \cup \set{v_⊥}\\
    E' &\coloneq E \cup
    \set[(v_p,v)]{\text{$p$ is a priority of $P$}} \cup {}\\
    &\set[(u,v_p)]{\text{$p$ is odd and $(u,p-1) ∈ \strategyprofile(\pi,v)$}} \cup {}\\
    &\set[(u,v_p)]{\text{$p$ is even and $(u,p+1) ∈ \strategyprofile(\pi,v)$}} \cup {}\\
    &\set[(u,v_⊥)]{\text{$(u,p) ∉ \strategyprofile(\pi,v)$ for all $p$}}\\
    \omega'(w) &\coloneq
    \begin{cases}
      \omega(w) & \text{if $w ∈ V(P)$}\\
      0 & \text{if $w = v_⊥$}\\
      p & \text{if $w = v_p$}.
    \end{cases}
  \end{align*}

  Note that for each $u ∈ U$, there is exactly one $v_p$ such that
  $(u,v_p) ∈ E'$.  So we can extend $\pi$ to a strategy $\pi'$ on $P'$
  by defining $\pi'(v_1,\ldots,u) = v_p$ if $\pi(v_1,\ldots,u)$ is
  undefined.

  We claim that $\pi'$ is a winning strategy.  Let $v=v_1,v_2,\ldots$
  be an infinite $\pi'$-conforming path.  Clearly the path is winning
  if it has a $\pi$-conforming suffix.

  So assume that it visits some $u ∈ U$ an infinite number of times
  followed by $v_p$.  If $p$ is odd, then $(u,p-1) ∈
  \strategyprofile(\pi,v)$ guarantees that the worst priority on all
  path segments that go from $v$ to $u$ is $p-1$.  By the pigeon
  principle there is at least one priority $p' \sqsubseteq p-1$ that
  we visit infinitely often on the path.  Furthermore, $p' \leq p-1$
  because $p-1$ is even.  This means the priority $p$ of $v_p$ is
  irrelevant because $p' < p$.

  If $p$ is even, then $(u,p+1) ∈ \strategyprofile(\pi,v)$ guarantees
  that the worst priority on all path segments that go from $v$ to $u$
  is $p+1$.  So there must be a minimum priority $p' \sqsubseteq p+1$
  that occurs infinitely often on these path segments.  If $p' \geq
  p$, then $p'$ becomes irrelevant because we visit $v_p$ an infinite
  number of times.  If $p' < p$, then $p$ becomes irrelevant.
  However, $p' \sqsubseteq p+1$ then implies that $p'$ is even.

  We repeat this argument for all pairs $(u,v_p)$ that occur
  infinitely often in the path.  We see that in all cases the minimum
  priority that occurs infinitely often is even, so $\pi'$ is a
  winning strategy.

  By \cref{thm:positional_strategies_exist}, there exists a positional
  winning strategy $\rho'$ on $(P',v)$.  Let $\rho$ be the restriction
  of $\rho'$ to $P$.  We claim that $\strategyprofile(\rho,v)
  \sqsubseteq \strategyprofile(\pi,v)$.

  Clearly $(u,p) ∉ \strategyprofile(\pi,v)$ implies $(u,p) ∉
  \strategyprofile(\rho,v)$ because otherwise we would visit the node
  $v_⊥$ and immediately lose.  Let $(u,p) ∈ \strategyprofile(\rho,v)$
  and $(u,p') ∈ \strategyprofile(\pi,v)$.  We have to show $p
  \sqsubseteq p'$.  If $p \sqsupset p'$, then there is a
  $\rho$-conforming path from $v$ to $u$ with a priority no better
  than $p$.  In $P'$ this gives us a $\rho'$-conforming path by going
  back from $u$ to $v$.  However, the only new node we visit is
  $v_{p''}$ and $p''$ is not enough to offset $p$, so this path loses,
  contradicting the fact that $\rho'$ was a winning strategy.
\end{proof}

\begin{definition}
  The type of a node $v ∈ V(P)$ is the set of optimal profiles.
  \begin{align*}
    &\strategytype{P}{v} \coloneq \{ \strategyprofile(\pi,v) \mid {}\\
    &\qquad\quad \text{$\pi$ is a partial winning strategy for $(P,v)$ and} \\
    &\qquad\quad \text{there is no partial winning strategy $\pi'$ such that}\\
    &\qquad\quad\quad \text{$\strategyprofile(\pi',v) \sqsubset
      \strategyprofile(\pi,v)$} \}.
  \end{align*}
\end{definition}
By \cref{lem:positional_partial_strategies_exist}, the strategies
occurring in the above definition can be chosen to be positional.

Next, we define the notion of a parity game \emph{simulating} another
parity game.  A game simulates another game if it behaves in the same
way when viewed from the outside.  For every node in the old game
there must be a node in the new game that has the same type.
Internally the games could be quite different, and in fact the new
game could have a very different number of nodes than the old game.

Our goal is to find small games that simulate large games.

\begin{definition}
  Let $P, P'$ be partial parity games with the same interface $U$.

  The game $P'$ \emph{simulates} $P$ if there is a map $f: V(P) \to
  V(P')$ such that $f(u) = u$ for all $u ∈ U$ and for every node $v ∈
  V(P)$, $\strategytype{P}{v} = \strategytype{P'}{f(v)}$.
\end{definition}

Whenever we have a game $P$ with an induced subgame $Q$ with no edges
going from $Q$ to the rest of $P$ except via the interface of $Q$, we
can replace $Q$ in $P$ by one of its simulations without the rest of
$P$ noticing.

\begin{lemma}[Simulation Lemma]\label{lem:simulation_lemma}
  Let $P,Q$ be parity games such that $Q$ is an induced subgame of $P$
  with interface $U$ and with no edges from $Q \setminus U$ to $P
  \setminus Q$.  Let $Q'$ be a partial parity game with interface $U$
  which simulates $Q$ via the function $f: V(Q) \to V(Q')$.  Extend
  $f$ to $V(P)$ by letting $f(v) = v$ for all $v ∈ V(P) \setminus
  V(Q)$.

  Define $P'$ as the parity game where the induced subgame $Q$ has
  been replaced by $Q'$ and edges pointing to nodes $v ∈ V(Q)$ now
  point to $f(v) ∈ V(Q')$.

  Then for all $v ∈ V(P)$, \peven wins $(P,v)$ iff \peven wins
  $(P',f(v))$.
\end{lemma}
\begin{proof}
  Translation of strategies.  Because the types agree, neither player
  can be worse off in one game.
\end{proof}

\subsection{Definable Profiles}
\label{sec:definable_profiles}

In the next step, we would like to encode a profile in a formula.
Given a profile $y$ in a model checking game and a starting point $x =
(x',\psi)$, we would like to define a formula $ψ^y$ with the property
that $ψ^y$ is true on the node $x'$ in the structure if and only if
the profile $y$ is possible on $(P,x)$.  However, we do not know how
to do this.

Hence we weaken the restriction and want $ψ^y$ to be true iff a
profile $y' \sqsubseteq y$ is possible.  This is enough for our
purposes because the type of $x$ only cares about
$\sqsubseteq$-minimal profiles.  This formula turns out to be
definable.  Using a suitable definition of $ψ^y$, we get the following
theorem.

\begin{theorem}
  \label{thm:definable_profiles}
  Let $\seq P$ be a sequence of proposition symbols disjoint from
  $\sigma$.  Let $φ ∈ L_μ[\sigma \cup P]$, $M, v$ be a
  $\sigma$-structure and $\seq X$ be a sequence of nodes of $M$.  For
  $\ind{ψ} ∈ \CL(φ)$, $y ∈ \strategyprofiles(\MCgame{M}{φ})$, it holds
  that $M,v ⊧ \defineprofile{\ind{ψ}}{y}$ iff there is a positional
  partial winning strategy $\pi$ for $(\MCgame{M}{φ},(v,\ind{ψ}))$
  such that $\strategyprofile(\pi,(v,\ind{ψ})) \sqsubseteq y$.
\end{theorem}

\begin{corollary}
  \label{cor:type_determined_by_type}
  Let $\seq P$ be a sequence of proposition symbols disjoint from
  $\sigma$.  Let $φ ∈ L_μ[\sigma \cup P]$, $M,v$ be a
  $\sigma$-structure, $\seq X ∈ V(M)^{\abs{P}}$ and $\ind{ψ} ∈
  \CL(φ)$.  Then
  \begin{align*}
    \strategytype{\MCgame{M}{φ}}{(v,\ind{ψ})} &=
    \Bigl\{{y ∈ \strategyprofiles(\MCgame{M}{φ})} \mathrel{}\Big|\mathrel{}\\
    & \hspace{-2.2cm}
      \text{$M,v ⊧ \defineprofile{\ind{ψ}}{y}$ and there is no $y'
        \sqsubset y$ with $M,v ⊧ \defineprofile{\ind{ψ}}{y'}$}\Bigr\}.
  \end{align*}
  That is, $\ctypeWithIndex{\set{φ}}{\seq P}{M}{v}{\seq X}$ determines
  $\strategytype{\MCgame{M}{φ}}{(v,\ind{ψ})}$.
\end{corollary}

Before we can explain $ψ^y$, we need one more definition.
\begin{definition}
  For an annotated $φ ∈ L_μ[\sigma]$, $\ind{ψ} ∈ \CL(φ)$ and
  $\ind{\chi} ∈ \sub(ψ)$, let $\prio{φ}{\ind{ψ}}{\ind{\chi}}$ be the
  minimum priority of all fixpoint operators that enclose $\ind{\chi}$
  in $\psi$.
\end{definition}
\begin{definition}\label{def:psi_y}
  Let $\seq P = (P_1,\ldots,P_k)$ be a sequence of proposition symbols
  disjoint from $\sigma$.  Let $φ ∈ L_μ[\sigma \cup P]$ be a formula,
  $M$ be a $\sigma$-structure and $X = (x_1,\ldots,x_k) ∈ V(M)^k$.
  Let $\ind{ψ} ∈ \CL(φ)$ and $y ∈ \strategyprofiles(\MCgame{M}{φ})$.
  For every $\ind{ψ'} ∈ \sub^+(ψ)$, there is a formula $\ind{φ'} ∈
  \CL(φ)$ corresponding to $\ind{ψ'}$.  We inductively define an
  operation $\defineprofile{\cdot}{y}$ over the structure of $ψ'$.
  \begin{align*}
    \defineprofile{\ind{V}}{y} &\coloneq V,
    \quad\defineprofile{(\ind{\neg V})}{y} \coloneq \neg V &&
    \text{for prop.\ or var.\ $V$}\\
    \defineprofile{(\ind{\chi \ast \chi'})}{y} &\coloneq
    (\defineprofile{\chi}{y}) \ast (\defineprofile{\chi'}{y}) &&
    \text{for
      ${\ast} ∈ \set{{∨},{∧}}$}\\
    \defineprofile{(\ind{\alpha X.\chi})}{y} &\coloneq
    \alpha X.(\defineprofile{\chi}{y}) && \text{for $\alpha ∈ \set{μ,ν}$}\\
    \defineprofile{(\ind{◊\chi})}{y} &\coloneq
    \biggl(\Bigl(\bigvee_{i ∈ N} P_i\Bigr) ∨ ◊(\defineprofile{\chi}{y})\biggr)\\
    \defineprofile{(\ind{□\chi})}{y} &\coloneq
    \biggl(\Bigl(\bigwedge_{i ∈ N} \neg P_i\Bigr) ∧
      □(\defineprofile{\chi}{y})\biggr)
  \end{align*}
  In the case $\ind{◊\chi}$, we use
  \begin{align*}
    N &\coloneq \{ 1 \leq i \leq k \mid {}\\
    &\qquad\text{$((x_{i},\ind{φ'}),p') ∈ y$ for some $p' \sqsupseteq
      \prio{φ}{ψ}{\ind{◊\chi}}$} \}.
  \end{align*}
  In the case $\ind{□\chi}$, we use
  \begin{align*}
    N &\coloneq \{ 1 \leq i \leq k \mid {}\\
    &\qquad\text{$((x_{i},\ind{φ'}),p') ∉ y$
      for all $p' \sqsubset \prio{φ}{ψ}{\ind{□\chi}}$} \}.
  \end{align*}
  In both cases, $\ind{φ'} ∈ \CL(φ)$ is the formula corresponding to
  $\ind{◊\chi}$ or $\ind{□\chi}$, respectively.
\end{definition}

The motivation behind this seemingly quite arbitrary definition is
that if a profile says we can reach $(x_i,\ind{◊\chi})$ with the worst
priority $p'$, and the actual priority we have is at least as good as
$p'$, we are allowed to take the shortcut and leave the game.  That is
why we add $X_i$ to the disjunction in this case.  Of course, we need
to pay close attention to the games that are involved, because
$(x_i,\ind{◊\chi})$ is not a node in $\MCgame{M}{φ}$ and $y$ is not a
profile of $\MCgame{M}{ψ}$.  However, this is not a problem because
every $\ind{◊\chi}$ corresponds to a unique $\ind{φ'} ∈ \CL(φ)$, and
the game $\MCgame{M}{ψ}$ is a partial unfolding of the
$\MCgame{M}{φ}$.  This means that every strategy on one of these games
is also a strategy on the other game, although not necessarily
positional.

Dually, in the case $\ind{□\chi}$, if the actual priority is worse
than what the profile wants, we must make sure that
$(x_i,\ind{□\chi})$ is not reached, so we add $\neg X_i$ with a
conjunction.

A formal statement of this explanation is
\cref{thm:definable_profiles}.  Before we can prove this, however, we
need a technical lemma about $\prio{φ}{\ind{ψ}}{\ind{\chi}}$.
\begin{lemma}
  \label{lem:prioprio}
  Let $M$ be a structure, $v,x ∈ V(M)$ and $φ ∈ L_μ$, $\ind{ψ} ∈
  \CL(φ)$ and $\ind{\chi} ∈ \sub(ψ)$.  Then every path from
  $(v,\ind{\psi})$ to $(x,\ind{\chi})$ in $\MCgame{M}{ψ}$ (with
  priorities according to $φ$) has $\prio{φ}{\ind{ψ}}{\ind{\chi}}$ as
  its minimum priority.
\end{lemma}
\begin{proof}
  Let $p$ be the minimum priority of a path from $(v,\ind{ψ})$ to
  $(x,\ind{\chi})$.  Clearly $p \leq \prio{φ}{\ind{ψ}}{\ind{\chi}}$
  because $\ind{\chi}$ is a subformula of $\ind{\psi}$, so every
  fixpoint operator enclosing $\ind{\chi}$ must have been visited
  at some point on the path.

  Assume to the contrary that $p < \prio{φ}{\ind{ψ}}{\ind{\chi}}$.
  This means that there is a node $(v',\ind{\priority{p}{\alpha}
    X.\psi'})$ with $\alpha ∈ \set{μ,ν}$ on the path.  Assume this is
  the first node of priority $p$ on the path.  The priorities increase
  with respect to a fixed sequence of variables $\seq Z$, so $\psi'$
  cannot contain a free variable $Y$ for any $Y$ that is quantified
  earlier, or $p$ would have to be larger.  But this means that
  $\ind{\priority{p}{\alpha} X.\psi'}$ is a closed formula.  So in
  order to reach $(x,\ind{\chi})$, the formula $\ind{\chi}$ must be a
  subformula of $\ind{\priority{p}{\alpha} X.\psi'}$, and we have that
  $\priority{p}{\alpha} X$ encloses $\ind{\chi}$, a contradiction to
  $p < \prio{φ}{\ind{ψ}}{\ind{\chi}}$.
\end{proof}

We split the proof of \cref{thm:definable_profiles} into two
directions.  \Cref{lem:winning_strategy_translation} shows the first
direction and \cref{lem:winning_strategy_translation_reversed} the
other.
\begin{lemma}
  \label{lem:winning_strategy_translation}
  Let $\seq P$ be a sequence of $k$~proposition symbols disjoint from
  $\sigma$.  Let $φ ∈ L_μ[\sigma \cup P]$, $M,v$ be a
  $\sigma$-structure, $\seq X ∈ V(M)^k$, $\ind{ψ} ∈ \CL(φ)$, $y ∈
  \strategyprofiles(\MCgame{M}{φ})$.  Let $\pi$ be a partial
  winning strategy for $(\MCgame{M}{ψ},(v,\ind{ψ}))$ and $\pi_φ$ be
  the corresponding strategy on $\MCgame{M}{φ}$.  If
  $\strategyprofile(\pi_φ,(v,\ind{ψ})) \sqsubseteq y$, then there
  exists a winning strategy $\pi'$ for
  $(\MCgame{M}{\defineprofile{\ind{ψ}}{y}},(v,\defineprofile{\ind{ψ}}{y}))$.
\end{lemma}
\begin{proof}
  Let $\pi$ be as required.  Without loss of generality we are going
  to assume that $\pi_φ$ is a positional strategy.  According to
  \cref{lem:positional_partial_strategies_exist}, this is always
  possible.  Then $\pi$ can be chosen to be positional, too.

  There is an obvious mapping from $\sub(\ind{ψ})$ to
  $\sub(\defineprofile{\ind{ψ}}{y})$ because
  $\defineprofile{\ind{ψ}}{y}$ is only a slightly modified version of
  $\ind{ψ}$.

  Define $\pi'$ positionally on
  $\MCgame{M}{\defineprofile{\ind{ψ}}{y}}$ so that it follows $\pi$
  wherever possible using the mapping we just described.  The only
  points where $\pi'$ is undefined are the nodes of the form
  $\overline{w} \coloneq (w,\ind{\bigvee_{i∈N} P_i ∨ ◊\chi})$.  On
  these nodes, if $w = x_i ∈ X$ for some $i ∈ N$, define
  $\pi'(\overline{w}) = (w,P_i)$.  Otherwise, define
  $\pi'(\overline{w}) = (w,\ind{◊\chi})$.  We claim that $\pi'$ is a
  strategy on
  $(\MCgame{M}{\defineprofile{\ind{ψ}}{y}},(v,\defineprofile{\ind{ψ}}{y}))$.

  Let $(x_i,\ind{◊\chi}) ∈ \MCgame{M}{\defineprofile{\ind{ψ}}{y}}$ be
  such that $(x_i,\ind{◊\chi})$ is reachable in
  $\MCgame{M}{\defineprofile{\ind{ψ}}{y}}^{\pi'}$ from
  $(v,\defineprofile{\ind{ψ}}{y})$ but $(x_i,\ind{◊\chi}) ∉
  \dom(\pi')$.  Let $(v_1,\ldots,v_n)$ be a $\pi'$-conforming path
  with $v_1 = (v,\defineprofile{\ind{ψ}}{y})$ and $v_n =
  (x_i,\ind{◊\chi})$ with minimum priority $p$.  This path corresponds
  to a $\pi$-conforming path in $\MCgame{M}{ψ}$ starting from
  $(v,\ind{ψ})$ with the same minimum priority, and hence a
  $\pi_φ$-conforming path in $\MCgame{M}{φ}$ with the same minimum
  priority.  By \cref{lem:prioprio}, we have $p =
  \prio{φ}{\ind{ψ}}{\ind{◊\chi}}$.

  Let $\ind{◊\chi'} ∈ \CL(φ)$ be the unique subformula of $φ$
  corresponding to $\ind{◊\chi}$.  Then we have
  $((x_i,\ind{◊\chi'}),p') ∈ \strategyprofile(\pi_φ,(v,\ind{ψ}))$ for
  some $p' \sqsupseteq p$ and hence
  $((x_i,\penalty0\ind{◊\chi'}),\penalty0p'') ∈ y$ for some $p''
  \sqsupseteq p'$.  By the construction of
  $\defineprofile{\ind{ψ}}{y}$, the node $(x_i,\ind{◊\chi})$ in
  $\MCgame{M}{\defineprofile{\ind{ψ}}{y}}$ must have a unique
  predecessor $(x_i,\ind{\bigvee_{i∈N} P_i ∨ ◊\chi})$ for some set
  $N$.  Recall the definition of $N$,
  \begin{align*}
    N &\coloneq \{ 1 \leq i \leq k \mid {}\\
    &\qquad\text{$(x_{i},\ind{◊\chi'},q) ∈ y$ for some $q \sqsupseteq
      \prio{φ}{ψ}{\ind{◊\chi}}$} \}.
  \end{align*}
  We find that $i ∈ N$, so the path was not $\pi'$-conforming.

  We need to show that $\pi'$ is winning.  Let $(v_1,\ldots,v_n)$ be a
  maximal $\pi'$-conforming path in
  $\MCgame{M}{\defineprofile{\ind{ψ}}{y}}$ with $v_1 =
  (v,\defineprofile{\ind{ψ}}{y})$.  By definition of $\pi'$, the last
  node cannot be $(w,P_i)$ for $w \neq x_i$.

  Assume $v_n ∈ V_◊$, that is, the path is losing.  The same path can
  be viewed as a maximal $\pi$-conforming path in $\MCgame{M}{ψ}$.  In
  $\MCgame{M}{ψ}$, the last node is also in $V_◊$ and has no
  successors, so we would have a $\pi$-conforming losing path in
  $\MCgame{M}{ψ}$, which contradicts the assumption that $\pi$ was a
  partial winning strategy.

  Clearly all infinite paths starting from
  $(\MCgame{M}{\defineprofile{\ind{ψ}}{y}},(v,\defineprofile{\ind{ψ}}{y}))$
  can never visit a node of the form $(w,P_i)$, so they can be viewed
  as paths on $\MCgame{M}{ψ}$.  They visit exactly the same
  priorities.  This implies that $\pi'$ is a winning strategy.
\end{proof}

\begin{lemma}
  \label{lem:winning_strategy_translation_reversed}
  Let $\seq P$ be a sequence of $k$~proposition symbols disjoint from
  $\sigma$.  Let $φ ∈ L_μ[\sigma \cup P]$, $M,v$ be a
  $\sigma$-structure, $\seq X ∈ V(M)^k$, $\ind{ψ} ∈ \CL(φ)$, $y ∈
  \strategyprofiles(\MCgame{M}{φ})$. Let $\pi'$ be a winning strategy
  for
  $(\MCgame{M}{\defineprofile{\ind{ψ}}{y}},(v,\defineprofile{\ind{ψ}}{y}))$.

  Then there exists a partial winning strategy $\pi$ for
  $(\MCgame{M}{ψ},(v,\ind{ψ}))$ such that for the corresponding
  partial strategy $\pi_φ$ on $\MCgame{M}{φ}$ it holds that
  $\strategyprofile(\pi_φ,(v,\ind{ψ})) \sqsubseteq y$.
\end{lemma}
\begin{proof}
  Let $\pi'$ be as required.  Assume $\pi'$ is a positional winning
  strategy.  Define $\pi$ (positionally) like $\pi'$ where possible.
  If $\pi'((w,\ind{\bigvee_{i∈N} P_i ∨ ◊\chi})) = (x_i,P_i)$ for some
  $\overline{w} ∈ V(\MCgame{M}{\defineprofile{\ind{ψ}}{y}})$, $N$ and
  $i$, then leave $\pi'((x_i,\ind{◊\chi}))$ undefined.

  Similar to the proof of the previous lemma one shows that $\pi'$ is
  a partial winning strategy and
  $\strategyprofile(\pi_φ,(v,\ind{ψ})) \sqsubseteq y$.
\end{proof}

\subsection{A Small Parity Game}
\label{sec:canonical_game}

With \cref{thm:definable_profiles} at our hands, we can now define a
partial parity game simulating the model checking game that only
depends on the types of some nodes in the original structure.  The
parity game consists of four layers of nodes.

\begin{enumerate}
\item One layer of ◊-nodes, one for each type, where \peven can choose
  a profile.
\item Then one layer of □-nodes, one for each profile, where \podd can
  choose one of the allowed paths.
\item Then a layer of nodes with out-degree~1 to ensure the priorities
  match the chosen path.
\item Finally a layer representing the interface.
\end{enumerate}

The edges only point from one layer to the next or from the last layer
back to the first layer.  Formally, let $M$ be a structure and $X =
\set{x_1,\ldots,x_k} ⊆ V(M)$.  Let $φ ∈ L_μ$.  First, we define the
layers described above.
\begin{align*}
  V_1 &\coloneq 2^{\strategyprofiles(\MCgame{M}{φ})}&
  V_3 &\coloneq \strategytargets(\MCgame{M}{φ})\\
  V_2 &\coloneq \strategyprofiles(\MCgame{M}{φ})&
  V_4 &\coloneq X \times \CL(φ).
\end{align*}

Next, we define the game $P^φ = (V, V_◊, E, \omega)$ with interface
$V_4$ depending only on $φ$ and the sets
$\ctypeWithIndex{\set{φ}}{\seq P}{M}{x_i}{\seq X}$, but not on $M$.
\begin{align*}
  V &\coloneq V_1 \cup V_2 \cup V_3 \cup V_4 \qquad E \coloneq E_1
  \cup E_2 \cup E_3 \cup E_4\\
  V_◊ &\coloneq V_1 \cup \set[(x_i,\ind{ψ}) ∈ V_4]{\text{$ψ$ starts with a ◊} }\\
  \omega(v) &\coloneq \begin{cases}
    p & \text{for $v = (x_i,\ind{ψ},p) ∈ V_3$}\\
    p' & \text{otherwise},
  \end{cases}
\end{align*}
where $p'$ is the maximum priority of $φ$.

For the set of edges, we connect the nodes according to the subset
relation and the nodes from $V_4$ back to their types.
\begin{align*}
  E_1 &\coloneq \set[(x,y) ∈ V_1 \times V_2]{y ∈ x}\\
  E_2 &\coloneq \set[(x,y) ∈ V_2 \times V_3]{y ∈ x}\\
  E_3 &\coloneq \set[\bigl((x_i,\ind{ψ},p), (x_i',\ind{ψ'})\bigr) ∈
  V_3 \times V_4]{\text{$(x_i,\ind{ψ}) = (x_i',\ind{ψ'})$}}\\
  E_4 &\coloneq \set[{(x,t) ∈ V_4 \times V_1}]{t =
    \strategytype{\MCgame{M}{φ}}{x}}.
\end{align*}

Note that $E_4$ is determined by the sets
$\ctypeWithIndex{\set{φ}}{\seq P}{M}{x}{\seq X}$ by
\cref{cor:type_determined_by_type}.

\begin{figure}
  \centering
  \begin{tikzpicture}[xscale=2.5,yscale=1.5]
    \tikzstyle{every node}=[circle,draw=black,thick,inner sep=3pt,minimum size=2mm]
    \tikzstyle{every path}=[->,thick,line width=1pt]
    \tikzstyle{intersection1}=[fill=green,thick]
    \tikzstyle{intersection2}=[fill=red]
    \tikzstyle{boxplayer}=[rectangle]
    \node [ellipse] (A) at (0,0) {$\strategytype{\MCgame{M}{φ}}{v}$};
    \node [boxplayer] (B1) at (-1,-1) {$\set{(v_1,1)}$};
    \node [boxplayer,align=left] (B2) at (0,-1) {$\{(v_1,0),$\\$(v_2,2),$\\$(v_4,0)\}$};
    \node [boxplayer] (B3) at (1,-1) {$\set{(v_2,3)}$};
    \node [boxplayer] (P1) at (-1.5,-2) {$z_1,1$};
    \node [boxplayer] (P2) at (-1,-2) {$z_1,0$};
    \node [boxplayer] (P3) at (-0.5,-2) {$z_2,2$};
    \node [boxplayer] (P4) at (1.5,-2) {$z_4,0$};
    \node [boxplayer] (P5) at (0.3,-2) {$z_2,3$};
    \node [ellipse] (X1) at (-1.5,-2.7) {$z_1$};
    \node [boxplayer,inner sep=5pt] (X2) at (-0.5,-2.7) {$z_2$};
    \node [ellipse] (X3) at (0.5,-2.7) {$z_3$};
    \node [boxplayer,inner sep=5pt] (X4) at (1.5,-2.7) {$z_4$};
    \draw (A) -- (B1);
    \draw (A) -- (B2);
    \draw (A) -- (B3);
    \draw (B1) -- (P1);
    \draw (B2) -- (P2);
    \draw (B2) -- (P3);
    \draw (B2) -- (P4);
    \draw (B3) -- (P5);
    \draw (P1) -- (X1);
    \draw (P2) -- (X1);
    \draw (P3) -- (X2);
    \draw (P4) -- (X4);
    \draw (P5) -- (X2);
  \end{tikzpicture}
  \caption{A part of $P^φ$}
  \label{fig:simulation_game}
\end{figure}

To illustrate this construction, assume that $\MCgame{M}{φ}$ has the
interface $\set{z_1,\ldots,z_4} ∈ X \times \CL(φ)$ and a node $v ∈
\MCgame{M}{φ}$ with $\strategytype{\MCgame{M}{φ}}{v} =
\bigl\{\set{(z_1,1)},\penalty0\{(z_1,0),(z_2,2),\penalty0(z_4,0)\},
\set{(z_2,3)}\bigr\}$.  \Cref{fig:simulation_game} illustrates a part
that could occur in the game $P^φ$.  In the full game $P^φ$, we would
also add the edges $(z_i, \strategytype{\MCgame{M}{φ}}{z_i}) ∈ V_4
\times V_1$.  In the node $\strategytype{\MCgame{M}{φ}}{v}$, \peven
can choose one of the possible profiles.  This corresponds to \peven
fixing a strategy $\pi$.  After fixing her strategy, \podd can choose
a path through the game conforming to this strategy.  The profile
tells us exactly what the worst possible paths are, and the layer
$V_3$ makes sure that the correct priority is visited.

The goal of this construction is to get a game such that the type of a
node labeled $\strategytype{\MCgame{M}{φ}}{v}$ is exactly
$\strategytype{\MCgame{M}{φ}}{v}$.  This leads to the main theorem of
this subsection.
\begin{theorem}
  \label{thm:small_simulation}
  For a formula $φ ∈ L_μ$, a structure $M$ and $X ⊆ V(M)$, the game
  $P^φ$ simulates $\MCgame{M}{φ}$.
\end{theorem}
\begin{proof}
  For every node $u ∈ X \times \CL(φ)$, define $f(u) = u$.  For the
  remaining nodes $v ∈ V(\MCgame{M}{φ}) \setminus (X \times \CL(φ))$,
  define $f(v) = \strategytype{\MCgame{M}{φ}}{v} ∈ V_◊(P^φ)$.

  All we have to do now is to show that
  $\strategytype{\MCgame{M}{φ}}{v} = \strategytype{P^φ}{f(v)}$ for all
  $v ∈ V(\MCgame{M}{φ})$.  First we show $⊆$.

  Let $\pi$ be a positional partial winning strategy for
  $(\MCgame{M}{φ},(v,\ind{ψ}))$.  We want to construct a positional
  partial winning strategy $\pi'$ for $(P^φ,f((v,\ind{ψ})))$ such that
  $\strategyprofile(\pi,(v,\ind{ψ})) =
  \strategyprofile(\pi',f((v,\ind{ψ})))$.

  For every node $(v,\ind{ψ}) ∈ \MCgame{M}{φ}$, define
  \[ \pi'(\strategytype{\MCgame{M}{φ}}{(v,\ind{ψ})}) \coloneq
  \strategyprofile(\pi,(v,\ind{ψ})). \] For $(x_i,\ind{ψ}) ∈
  V_◊(P^φ)$, if $(x_i,\ind{ψ}) ∈ \dom(\pi)$, then we define
  $\pi'((x_i,\ind{ψ})) = \strategytype{\MCgame{M}{φ}}{x_i}$.
  Otherwise, leave $\pi'((x_i,\ind{ψ}))$ undefined.

  We claim that $\pi'$ is a partial winning strategy on
  $(P^φ,(v,\ind{ψ}))$.  By \cref{thm:definable_profiles}, for all
  $(x_i,\ind{\chi}) ∈ X \times \CL(φ)$ it holds that $M,x_i ⊧
  \defineprofile{\ind{\chi}}{\strategyprofile(\pi,(x_i,\ind{\chi}))}$.
  So the unique edge leaving from $(x_i,\ind{ψ})$ in $P^φ$ goes to
  some node $y$ with $\strategyprofile(\pi,(x_i,\ind{ψ})) ∈ y$.

  Inductively it follows that every $\pi'$-conforming path in $P^φ$
  corresponds to a $\pi$-conforming path in $\MCgame{M}{φ}$ and vice
  versa.  So $\pi'$ is a partial winning strategy with
  $\strategyprofile(\pi,(v,\ind{ψ})) =
  \strategyprofile(\pi',f((v,\ind{ψ})))$.

  It remains to show the other direction
  $\strategytype{\MCgame{M}{φ}}{v} \supseteq
  \strategytype{P^φ}{f(v)}$.

  Let $\pi'$ be a positional partial winning strategy for
  $(P^φ,f((v,\ind{ψ})))$.  We want to construct a partial winning
  strategy $\pi$ for $(\MCgame{M}{φ},(v,\ind{ψ}))$ such that
  $\strategyprofile(\pi,\penalty0(v,\ind{ψ})) =
  \strategyprofile(\pi',f((v,\ind{ψ})))$.

  By \cref{thm:definable_profiles} and some technical work, we can
  show that there is a $\pi$ such that
  $\strategyprofile(\pi,(v,\ind{ψ})) \sqsubseteq
  \strategyprofile(\pi',f((v,\ind{ψ})))$.  As we saw when proving the
  other direction, we can construct from $\pi$ a partial winning
  strategy $\pi''$ for $(P^φ,f((v,\ind{ψ})))$ such that
  $\strategyprofile(\pi,(v,\ind{ψ})) =
  \strategyprofile(\pi'',f((v,\ind{ψ})))$.  From the definition of
  $\strategytype{}{}$ it follows that
  $\strategyprofile(\pi,(v,\ind{ψ})) =
  \strategyprofile(\pi',(v,\ind{ψ}))$.
\end{proof}

\subsection{Proof of the Decomposition Theorem}

With \cref{thm:small_simulation}, we finally have the necessary tool
to conclude the proof of the decomposition theorems from
\cpageref{thm:types_suffice,thm:computing_types}.

\smallskip
\begin{proof}[of \cref{thm:types_suffice}]
  Fix some $φ ∈ \CL_P(L)$.  Consider the model checking game
  $\MCgame{M}{φ}$ and the induced subgame $\MCgame{M_2}{φ}$ with
  interface $U$.  We can assume that $V(\MCgame{M_2}{φ}) \cap
  V(\MCgame{M}{φ}) = U$ by duplicating some nodes as necessary.

  The game $\MCgame{M_2}{φ}$ is simulated by $P^φ$, constructed as
  described in \cref{thm:small_simulation}.  By
  \cref{lem:simulation_lemma}, we can replace $\MCgame{M_2}{φ}$ by
  $P^φ$ (by properly adapting the edges) without changing the winner
  on $(v,φ)$.  Since the construction of $P^φ$ only depends on the
  types of the nodes in $X$, we will get the same game $P^φ$ if we
  start the construction with $M_2'$.

  Let $(v,w)$ be an edge from $M_1 \setminus X$ to $M_2 \setminus X$
  and let $w' ∈ M_2'$ be the node chosen as the replacement for $w$.
  Because $\ctypeWithIndex{\set{φ}}{\seq P}{M_2}{w}{\seq X}$
  determines $\strategytype{\MCgame{M_2}{φ}}{(w,φ)}$ by
  \cref{cor:type_determined_by_type} and we have
  \[ \ctypeWithIndex{\set{φ}}{\seq P}{M_2}{w}{\seq X} \subseteq
  \ctypeWithIndex{L}{\seq P}{M_2}{w}{\seq X}, \] it follows that
  \[ \strategytype{\MCgame{M_2}{φ}}{(w,φ)} =
  \strategytype{\MCgame{M_2'}{φ}}{(w',φ)}. \] So in the simulation,
  the edge will point to the same node no matter if we started with
  $M_2$ or $M_2'$.
\end{proof}

\begin{proof}[of \cref{thm:computing_types}]
  The first part is essentially a different way of stating
  \cref{thm:types_suffice} which follows immediately with the same
  argument as in the previous proof.

  Note that we may assume without loss of generality that $X \cap Y =
  \emptyset$.  If this is not the case, then we have $x_i = y_j$ for
  some $x_i ∈ X, y_j ∈ Y$ and the propositional variables $X_i ∈ P$
  and $Y_j ∈ Q$ will be interchangeable.

  Set $L' \coloneq \CL_Q(L)$.  \Cref{thm:types_suffice} states that
  $\ctypeWithIndex{L'}{\emptyset}{M}{v}{\emptyset}$ is invariant under
  $\CL_\emptyset(L')$-equivalent directed separations for all $v ∈
  M_1$.  All we need to show is that the requirements listed in
  \cref{thm:computing_types} specify the directed separation
  $(M_1,M_2)$ up to $\CL_\emptyset(L')$-equivalence.

  For all nodes $w ∈ M_2$, the set $\ctypeWithIndex{L'}{\seq
    P}{M_2}{w}{\seq X}$ can be computed from $\ctypeWithIndex{L}{\seq
    P}{M_2}{w}{\seq X}$; a propositional variable $Y_i ∈ Q$
  corresponding to a node $y_i ∈ Y$ is always false in $M_2$.  From
  this we can easily compute
  $\ctypeWithIndex{L'}{\emptyset}{M_2}{w}{\emptyset}$ by forgetting
  about $\seq P$.

  The computability in the above argument follows from the observation
  that all sets involved are finite in size and the model checking for
  $L_μ$ is decidable.

  For the second part, let $φ ∈ \CL_Q(L)$.  We want to decide whether
  $M,w ⊧ φ$.  By the first part, we already know the sets
  $\ctypeWithIndex{L}{\seq Q}{M}{x_i}{\seq Y}$ for all $x_i ∈ X$.
  Consider the model checking game $\MCgame{M}{φ}$.  In this game, the
  nodes of the form $(v,Y_i)$ with $v ∈ M_1$ are always losing because
  $Y_i ∈ Q$ is never true in $M_2$.  It follows that the subgame
  $\MCgame{M_2}{φ}$ is isomorphic to $\MCgame{M_2}{φ'}$, where $φ'$ is
  constructed from $φ$ by replacing all $Y_i ∈ Q$ by $\bot$.  Note
  that $φ' ∈ \CL_P(L)$, so we know all optimal partial strategies for
  $(\MCgame{M_2}{φ'},(w,φ'))$ because we know $\ctypeWithIndex{L}{\seq
    P}{M_2}{w}{\seq X}$.  It follows that the winner is determined by
  the remaining sets given in the theorem.
\end{proof}

%%% Local Variables:
%%% TeX-master: "types"
%%% End:

\section{FPT Algorithms for \texorpdfstring{$L_\mu$}{L\_mu} Model
  Checking}
\label{sec:model-checking}

In this section we derive two algorithmic applications of
\cref{thm:computing_types}.  More precisely, we show that
$L_\mu$-model-checking is fixed-parameter tractable on any class of
structures of bounded Kelly-width or bounded DAG-width.

Before proving our results, we develop some algorithmic concepts
common to both proofs.  We first need an algorithmic version of
$L$-equivalence.

In the following, let $\sigma$ be a signature, $\seq P$ be a sequence
of propositional symbols of the appropriate length disjoint from
$\sigma$ and let $L ⊆ L_μ[\sigma \cup P]$.

\subsection{Weak Separations}

\begin{definition}\label{def:weak-dir-separation}
  Let $M$ be a $\sigma$-structure.  A pair $(M_1,M_2)$ of induced
  substructures is a \emph{weak} directed $\sigma$-separation of $M$
  with interface $\seq X = (x_1,\ldots,x_k)$ if
  \begin{itemize}
  \item $V(M) = V(M_1) \cup V(M_2)$,
  \item $X = \set{x_1,\ldots,x_k} ⊆ V(M_1) \cap V(M_2)$,
  \item there are no edges from $M_2 \setminus (V(M_1) \cap
    V(M_2))$ to $M_1 \setminus (V(M_1) \cap V(M_2))$,
  \item there are no edges from $(V(M_1) \cap V(M_2)) \setminus X$ to
    $V(M_1) \setminus V(M_2)$.
  \end{itemize}
\end{definition}
Clearly, every directed separation is a weak directed separation.
Weak separations can be transformed into proper separations by
duplicating the nodes outside of the interface $X$.  This gives us the
following theorem.

\begin{theorem}\label{thm:weak-sep-avoidable}
  Let $(M_1,M_2)$ be a weak directed separation of $M$ with interface
  $\seq X$.  Then there exists a structure $M'$ and a directed
  separation $(M_1',M_2')$ of $M'$ with the same interface $\seq X$
  and isomorphisms $\pi_1: M_1 \to M_1'$, $\pi_2: M_2 \to M_2'$ which
  are the identity on $\seq X$ such that
  \begin{align*}
    \ctypeWithIndex{L}{\seq P}{M}{v}{\seq X} =
    \ctypeWithIndex{L}{\seq P}{M'}{\pi_i(v)}{\seq X}
  \end{align*}
  for all $i ∈ \set{1,2}$ and $v ∈ V(M_i)$.
\end{theorem}
\begin{proof}
  For $i ∈ \set{1,2}$, define $\pi_i$ and $M'$ as
  \begin{align*}
    \pi_i(v) &\coloneq
    \begin{cases}
      v & \text{if $v ∈ X$}\\
      (i,v) & \text{if $v ∉ X$}
    \end{cases}\\
    V(M') &\coloneq \pi_1(V(M_1)) \cup \pi_2(V(M_2))\\
    E(M') &\coloneq E_1 \cup E_2 \cup E_3,
  \end{align*}
  where, for $i ∈ \set{1,2}$,
  \begin{align*}
    E_i &\coloneq \{(\pi_i(v),\pi_i(w)) \mid (v,w) ∈ E(M_i) \}\\
    E_3 &\coloneq \{(\pi_1(v),\pi_2(w)) \mid {}\\
    &\qquad (v,w) ∈ E(M) \cap (V(M_1) \times V(M_2)) \}.
  \end{align*}
  The substructures $M_1'$, $M_2'$ of $M'$ are induced by the sets
  \begin{align*}
    V(M_i') &\coloneq \pi_i(V(M_i)).
  \end{align*}
  Clearly, $\pi_i$ is an isomorphism between $M_i$ and $M_i'$ and the
  identity on $\seq X$.  We also have that $(M_1',M_2')$ is a directed
  separation of $M'$.

  It is easy to verify that the colored structures $\colored{M}{\seq
    P}{\seq X}$ and $\colored{M'}{\seq P}{\seq X}$ are bisimilar.
  Bisimilarity of these structures implies
  \begin{align*}
    \ctypeWithIndex{L}{\seq P}{M}{v}{\seq X} =
    \ctypeWithIndex{L}{\seq P}{M'}{\pi_i(v)}{\seq X}
  \end{align*}
  for all $i ∈ \set{1,2}$ and $v ∈ V(M_i)$.
\end{proof}

Having isomorphisms means that
\begin{align*}
  \ctypeWithIndex{L}{\seq P}{M_i}{v}{\seq X} =
  \ctypeWithIndex{L}{\seq P}{M_i'}{\pi_i(v)}{\seq X}
\end{align*}
for all $v ∈ V(M_i)$.

This and the previous theorem imply that \cref{thm:types_suffice} and
with appropriate wording also \cref{thm:computing_types} hold for weak
directed separations as well.  Let us restate the last theorem in its
more general form.

\begin{theorem}[Corollary of
  \cref{thm:computing_types,thm:weak-sep-avoidable}]
  \label{thm:computing_types_weak}
  Let $\seq P$, $\seq Q$ be sequences of proposition symbols such that
  $\sigma \cap P = \sigma \cap Q = P \cap Q = \emptyset$.

  Let $L ⊆ L_μ[\sigma \cup P]$ and $M$ be a structure with a weak
  directed $\sigma$-separation $(M_1,M_2)$ with interface $\seq X$.
  Let $\seq Y ∈ ((V(M_1) \setminus V(M_2)) \cup X)^{\abs{Q}}$ be a
  tuple.

  For all $v ∈ M_1$, the set $\ctypeWithIndex{L}{\seq Q}{M}{v}{\seq
    Y}$ depends only on
  \begin{itemize}
  \item $M_1$ and $\seq Q$ and
  \item $\set[{(x_i,\ctypeWithIndex{L}{\seq P}{M_2}{x_i}{\seq
        X})}]{x_i ∈ X}$ and
  \item $\set{ (v,\ctypeWithIndex{L}{\seq P}{M_2}{w}{\seq X}) \mid
      \text{$(v,w) ∈ E(M) \cap (M_1 \times M_2)$} }$.
  \end{itemize}
  Provided $L$ is finite, $\ctypeWithIndex{L}{\seq Q}{M}{v}{\seq Y}$
  can be computed from these sets.

  Furthermore, for every $w ∈ M_2$, the set $\ctypeWithIndex{L}{\seq
    Q}{M}{w}{\seq Y}$ depends only on the above sets and on
  $\ctypeWithIndex{L}{\seq P}{M_2}{w}{\seq X}$ and can be computed
  from these sets if $L$ is finite.
\end{theorem}

The only difference of this statement to \cref{thm:computing_types} is
that we only require a weak separation and that the tuple $\seq Y$
should not contain a node $v ∈ V(M_1) \cap V(M_2)$ which is not part
of the interface.  This last requirement is necessary because
otherwise we would have a color in $M_2$ where there was none before,
and the types of $M_2$ with respect to $\seq X$ do not carry this
information.

\subsection{Kelly-Width}

First we consider Kelly-width.  We follow the notation and definitions
given in~\cite{hunter2008}.  For a directed acyclic graph (DAG), we
write $\preceq$ for the reflexive and transitive closure of the edge
relation.

Let $G$ be a digraph.  A set $W ⊆ V(G)$ \emph{guards} $X ⊆ V(G)$ if $W
\cap X = \emptyset$ and for all $(u,v) ∈ E(G)$ with $u ∈ X$, we have
$v ∈ X \cup W$. For any set $W\subseteq V(G)$ we write $\guard(W)$ for
the minimal set $U\subseteq V(G)$ guarding $W$.

\begin{definition}
  A \emph{Kelly decomposition} of a digraph $G$ is a triple $\DD
  \coloneq (D, \beta, \gamma)$, where $\beta, \gamma \st V(D)
  \rightarrow 2^{V(G)}$ such that
  \begin{itemize}
  \item $D$ is a DAG and $(\beta(t))_{t∈V(D)}$ partitions $V(G)$,
  \item for all $t ∈ V(D)$, $\gamma(t)$ guards $\Bdown_t \coloneq
    \bigcup_{t' \succeq t} \beta(t')$ and
  \item for all $s ∈ V(D)$ there is a linear order $\leq_t$ on its
    children so that the children can be ordered as $t_1, \ldots, t_p$
    such that for all $1 \leq i \leq p$, $\gamma(t_i) ⊆ \beta(s) \cup
    \gamma(s) \cup \bigcup_{j<i} \Bdown_{t_j}$.  Similarly, there is a
    linear order on the roots such that $\gamma(r_i) ⊆ \bigcup_{j<i}
    \Bdown_{r_j}$.
  \end{itemize}

  The \emph{width} of $\DD$ is $\max\set[\abs{\beta(t) \cup
    \gamma(t)}]{t ∈ V(D)}$.  The \emph{Kelly-width} of $G$ is the
  minimal width of any of its Kelly decompositions.
\end{definition}

\begin{theorem}\label{thm:lmu_fpt_on_bounded_kelly_width}
  There exists an algorithm that solves the $L_μ$ model checking
  problem in time $O(f(k + \abs{φ}) \cdot n^c)$ for some computable
  function $f$ and some constant $c$, where $k$ is the Kelly-width and
  $n$ the size of the input structure, provided a Kelly decomposition
  of width at most $k$ is given as part of the input.
\end{theorem}

Let $G$ be a structure of Kelly-width $k$ and $v\in V(G)$. It is
easily seen that, by increasing the Kelly-width by one, we can always
take a Kelly decomposition of $G$ of width $\leq k+1$ which has only
one root and this root contains $v$. We call such a Kelly
decomposition \emph{rooted at $v$}.

\smallskip
\begin{proof}[of \cref{thm:lmu_fpt_on_bounded_kelly_width}]
  Let $G,v$ be a structure and $\seq P$ be a sequence of $k$ fresh
  proposition symbols.  We pick an arbitrary linear order of $V(G)$ in
  order to define interfaces consistently.

  Let $\DD = (D,\beta,\gamma)$ be a Kelly decomposition of width $k$
  of $G$ rooted at $v$ and $\phi\in L_\mu$.  We set $L \coloneq \{
  \phi \}$.

  Let us introduce the abbreviation
  \begin{align*}
    \TTT(A,B) \coloneq \set{ (v,\ctypeWithIndex{L}{\seq P}{A}{v}{B})
      \mid v ∈ A}.
  \end{align*}

  We will inductively compute the types \[ \TTT(\Bdown_t \cup
  \gamma(t),\gamma(t)) \] for all $t ∈ V(D)$.  For the leaves, these
  sets can be computed by brute force.  Let $t ∈ V(D)$ be a node with
  children $s_1,\ldots,s_l$ and assume that we already know the above
  types for all $s_i$.

  Let
  \begin{align*}
    \delta_i &\coloneq \bigcup_{j \leq i} (\gamma(s_j) \cap (\beta(t)
    \cup \gamma(t)))\\
    \delta_i' &\coloneq \delta_i \cup \bigcup_{j \leq i} \Bdown_{s_j}.
  \end{align*}
  We inductively compute the types $\TTT(\delta_i',\delta_i)$.  For
  $i=1$ we already know these types by assumption.  Assume $i>1$.

  We want to construct weak directed separations.  Note that by
  assumption we know $\TTT(\Bdown_{s_i} \cup
  \gamma(s_i),\gamma(s_i))$. We now first compute \[ \TTT(\Bdown_{s_i}
  \cup \gamma(s_i) \cup \delta_{i-1},\gamma(s_i) \cup
  \delta_{i-1}). \] This is possible because
  $(\delta_{i-1},\Bdown_{s_i} \cup \gamma(s_i))$ is a directed
  separation with interface $\delta_{i-1} \cap \gamma(s_i)$.

  Next, we observe that $(\Bdown_{s_i} \cup \gamma(s_i) \cup
  \delta_{i-1}, \delta_{i-1}')$ is a weak directed separation with
  interface $\delta_{i-1} \cup (\gamma(s_i) \cap \delta_{i-1})$.  Thus
  \cref{thm:computing_types_weak} allows us to compute
  $\TTT(\delta_i',\delta_i)$.

  After the last step we still need to compute $\TTT(\Bdown_t \cup
  \gamma(t),\gamma(t))$ for the parent $t$.  The pair $(\beta(t) \cup
  \gamma(t),\delta_l')$ is a directed separation with interface
  $\delta_l$, which is the final piece to the proof.

  The runtime of this algorithm is $O(f(k+\abs{φ})\cdot n^3)$ for a
  function $f$ because $\abs{V(D)} \leq \abs{V(M)}$, and we consider
  every element $t ∈ V(D)$ at most once.  Every computation of
  $\TTT(\Bdown_t \cup \gamma(t),\gamma(t))$ requires time at most
  linear in $V(D)$ because $t$ has at most that many successors and at
  most quadratic in $V(M)$ because all sets involved are of size
  linear in $V(M)$.
\end{proof}

\subsection{DAG-width}

Next we consider DAG-width~\cite{BerwangerDHKO12}.

\begin{definition}
  A \emph{DAG decomposition} of a digraph $G$ is a pair $\DD \coloneq
  (D, (X_d)_{d∈D})$ such that
  \begin{itemize}
  \item $D$ is a DAG,
  \item $\bigcup_{d∈D} X_d = V(G)$,
  \item For all $d \preceq d' \preceq d''$, $X_d \cap X_{d''} ⊆
    X_{d'}$,
  \item for all edges $(d,d') ∈ E(D)$, $X_d \cap X_{d'}$ guards
    $X_{\succeq d'} \cap X_d$, where $X_{\succeq d'} = \bigcup_{d'
      \preceq d''} X_{d''}$.
  \end{itemize}

  The \emph{width} of $\DD$ is $\max\set[\abs{X_d}]{d ∈ V(D)}$.  The
  \emph{DAG-width} of $G$ is the minimal width of any of its DAG
  decompositions.
\end{definition}

\begin{theorem}\label{thm:lmu_fpt_on_bounded_dag_width}
  There exists an algorithm that solves the $L_μ$ model checking
  problem in time $O(f(k + \abs{φ}) \cdot n^c)$ for some computable
  function $f$ and some constant $c$, where $k$ is the DAG-width and
  $n$ the size of the input structure, provided a DAG decomposition of
  width at most $k$ is given as part of the input.
\end{theorem}
\begin{proof}
  Let $G,v_0$ be a structure and $(D,(X_d)_{d∈V(D)}$ be a nice DAG
  decomposition of $G$.  That means (see
  \cite{BerwangerDHKO12})
  \begin{enumerate}
  \item $D$ has a unique source.
  \item Every $d ∈ V(D)$ has at most two successors.
  \item For $d_0,d_1,d_2 ∈ V(D)$, if $d_1,d_2$ are two successors of
    $d_0$, then $X_{d_0} = X_{d_1} = X_{d_2}$.
  \item For $d_0,d_1 ∈ V(D)$, if $d_1$ is the unique successor of
    $d_0$, then $\abs{(X_{d_0} \setminus X_{d_1}) \cup (X_{d_1}
      \setminus X_{d_0})} = 1$.
  \end{enumerate}

  We set $L \coloneq \set{φ}$.  As in the proof for bounded Kelly
  width, we fix an arbitrary linear order $<$ on $V(G)$ so that we can
  consistently map nodes to the proposition symbols $P_i$ occurring in
  the types.

  During the run of the algorithm, we fill a table $\TTT$ with indices
  from the set $\set{ (v,d) ∈ V(G) \times V(D) \mid v ∈ X_{\succeq
      d}}$ and entries that are elements of $\allTypesWithIndex{\seq
    P}{L}$.  We will write to every index in this table at most once
  during the run, and we will always make sure to write
  \[ \TTT(v,d) = \ctypeWithIndex{L}{\seq P}{G[X_{\succeq
      d}]}{v}{X_d}. \] If $d$ is the root of $D$, then $\TTT(v_0,d)$
  will answer the model checking problem $G,v_0 ⊧ φ$.

  Clearly, we can fill in all values for the leaves $d$ immediately by
  computing them directly.

  Let $d ∈ V(D)$.  If $d$ has two successors $d_0,d_1$, then we have
  $X_d = X_{d_0} = X_{d_1}$.  Then $(G[X_{\succeq d_0]},G[X_{\succeq
    d_1}])$ is a weak directed separation with interface $X_d$.
  Because we already know $\ctypeWithIndex{L}{\seq P}{G[X_{\succeq
      d_i}]}{v}{X_d}$ for all $v$ and $i ∈ \set{0,1}$,
  \cref{thm:computing_types_weak} allows us to compute the types
  $\ctypeWithIndex{L}{\seq P}{G[X_{\succeq d}]}{v}{X_d}$.

  \newcommand{\shrink}{\mathrm{shrink}}

  The other case is that $d$ has a unique successor $d_0$.  Let $X_d =
  \set{v_1,\ldots,v_k}$ be ordered by the global linear order $<$.  If
  $X_{d_0} \setminus X_{d} = \set{v_i}$, then for all $v ∈ X_{\succeq
    X_d}$ we set
  \begin{align*}
    \TTT(v,d) &= \{ \shrink_i(ψ) \mid \\
    &\qquad\text{$ψ ∈ \TTT(v,d_0)$ and $P_i$ does not occur in $ψ$} \},
  \end{align*}
  where $\shrink_i(ψ)$ is a function defined inductively over the
  structure of formulas with the base case
  \begin{align*}
    \shrink_i(P_j) \coloneq
    \begin{cases}
      P_j & \text{if $j < i$}\\
      P_{j-1} & \text{if $j > i$}.
    \end{cases}
  \end{align*}
  In other words, $\shrink_i(ψ)$ is the formula $ψ$ with all $P_j$
  with $j>i$ replaced by $P_{j-1}$ in order to not leave a hole.  It
  is easy to check that we have
  \begin{align*}
    &\{ \shrink_i(ψ) \mid \text{$ψ ∈ \TTT(v,d_0)$ and $P_i$ does not occur in $ψ$} \}\\
    &\hspace{3cm}{}= \ctypeWithIndex{L}{\seq P}{G[X_{\succeq d}]}{v}{X_d}.
  \end{align*}

  The last case is $X_d \setminus X_{d_0} = \set{v_i}$.  Because $X_d
  \cap X_{d_0}$ guards $X_{\succeq d_0} \setminus X_d$, all edges
  $(v,v_i) ∈ G[X_{\succeq d}]$ satisfy $v ∈ X_d$.

  This means we have in fact a directed separation
  $(G[X_d],G[X_{\succeq d_0}])$ with interface $X_{d_0}$.  We know
  $G[X_d]$ (its size is small), and we know the types
  $\ctypeWithIndex{L}{\seq P}{G[X_{\succeq d_0}]}{v}{X_{d_0}}$ for all
  $v ∈ X_{\succeq X_{d_0}}$.

  By \cref{thm:computing_types}, this is all the information we need
  to compute $\ctypeWithIndex{L}{\seq P}{G[X_{\succeq d}]}{v}{X_d}$
  for all $v ∈ X_{\succeq X_d}$, which completes the algorithm and the
  proof.
\end{proof}

%%% Local Variables: 
%%% mode: latex
%%% TeX-master: "types"
%%% End: 

\section{Conclusion}
\label{sec:conclusion}

We proved a decomposition theorem for the modal μ-calculus.  This
theorem, interesting already all by itself, further allowed us to
prove fixed-parameter tractability results for the $L_μ$ model
checking problem on classes of bounded Kelly-width or bounded
DAG-width.

Open questions arise from the diverse number of decompositions for
directed graphs.  In particular, we think it could be promising to
analyze D-width~\cite{safari2005} and directed
tree-width~\cite{johnson2001}.

%%% Local Variables:
%%% TeX-master: "types"
%%% End:

\bibliographystyle{plain}
\bibliography{../refs.bib}

\end{document}